\documentclass[-bj]{imsart}
\usepackage{amsmath,amsthm,amssymb,natbib}
\usepackage{graphicx}
\RequirePackage[colorlinks,citecolor=blue,urlcolor=blue]{hyperref}

\startlocaldefs

\def\tilde{\widetilde}

\newcommand{\YY}{\mathbb{Y}}

\newcommand{\HH}{\mathbb{H}} 
\newcommand{\PP}{\mathbb{P}} 

\theoremstyle{localthm}
\theoremstyle{definition}

\newtheorem{Theorem}{Theorem}

\newtheorem{Corollary}[Theorem]{Corollary}
\newtheorem{Lemma}[Theorem]{Lemma}
\newtheorem{Remark}{Remark}

\newtheoremstyle{localrem}
	{5pt} 
	{5pt} 
	{\rm} 
	{} 
	{\bf} 
	{{\rm.}} 
	{.7em} 
	{} 

\theoremstyle{localrem}

\endlocaldefs

\begin{document}
\title{Hardy's Inequality and Its Descendants}
\runtitle{Hardy's inequalities} 

\begin{abstract}
We formulate and prove a generalization of Hardy's inequality \cite{Hardy:1925} in terms
of random variables and show that it contains the usual (or familiar)
continuous and discrete forms of Hardy's inequality.
Next we improve the recent version by \cite{MR4061539} of Hardy's inequality
with weights for general Borel measures and mixed norms so that it implies the discrete
version of  \cite{MR3405817} and the Hardy inequality with weights of
 \cite{MR0311856}
as well as the mixed norm versions due to
 \cite{MR1574995},  
 \cite{MR1574997},  
and  \cite{MR523580}. 
An equivalent formulation in terms of random variables is given
as well.
We also formulate a reverse version of Hardy's inequality, the closely
related Copson inequality, a reverse Copson inequality and a Carleman-P\'olya-Knopp inequality via random variables.
Finally we connect our Copson inequality with counting process martingales and survival analysis, and briefly discuss other applications.
\end{abstract}

\begin{aug}
\author{\fnms{Chris A. J.} \snm{Klaassen}\thanksref{a}\ead[label=e1]{C.A.J.Klaassen@uva.nl}} 
\and
\author{\fnms{Jon A.} \snm{Wellner}\thanksref{b}\corref{e2}\ead[label=e2]{jaw@stat.washington.edu}}

\address[a]{Korteweg-de Vries Institute for Mathematics, University of Amsterdam, The Netherlands \printead{e1}}
\address[b]{Statistics, Box 354322, University of Washington, Seattle, WA 98195-4322  \printead{e2}}
\runauthor{Klaassen and Wellner}
\end{aug}


\maketitle





\section{Introduction}\label{sec:I}

The classical Hardy inequality is often presented as the following  pair of
inequalities:
the {\sl continuous (or integral form) inequality} says,
if $p>1$ and $\psi$ is a nonnegative $p-$integrable function on $(0,\infty)$, then
\begin{eqnarray}
\int_0^\infty \left ( \frac{1}{x} \int_0^x \psi (y) dy \right )^p \, dx
\le  \left ( \frac{p}{p-1} \right )^p \int_0^{\infty} \psi^p(y)\, dy,
\label{ContHardyIneq}
\end{eqnarray}
while the {\sl discrete  (or series form) inequality} says,
if $p>1$ and  $\{ c_n \}_1^\infty$ is a sequence of nonnegative real numbers,
then
\begin{eqnarray}
\sum_{n=1}^\infty \left ( \frac{1}{n} \sum_{k=1}^n c_k \right )^p
\le \left ( \frac{p}{p-1} \right )^p  \sum_{k=1}^\infty c_k^p .
\label{DiscHardyIneq}
\end{eqnarray}
For example, see pp. 239--243 of \cite{MR0046395}, Exercises 3.14 and 3.15 of
\cite{MR0210528},   
 \cite{MR3676556},  
or
\cite{MR2062704}, Chapter 9.  

As \cite{Hardy:1925}  
mentions in his Section 5,
Landau pointed out that the discrete inequality follows from the integral one by noting that
$c_1 \geq c_2 \geq \cdots $ may be assumed, and by choosing an appropriate step function as
$\psi$; see also \cite{MR3676556}.

Our main objective here is to give a unified formulation and proof
of the inequalities (\ref{ContHardyIneq}) and (\ref{DiscHardyIneq})
using the notation and language of probability theory.
Along the way we will obtain a large family of other corollaries
related to weighted Hardy inequalities
(as given in \cite{MR2256532} and in the book-length treatments
\cite{MR3676556} and
\cite{MR2351524}); 
see Section \ref{sec:Mr}.

There is a vast literature on Hardy's inequality with weights with
\cite{MR0311856},  building on
\cite{MR280661} and  
\cite{MR0255751},  
as a milestone.
Versions of this inequality are useful in the study of
differential equations
(\cite{MR2508839}, 
\cite{MR3408787});  
the stability of stochastic processes
(\cite{MR3289377},   
\cite{MR1710983});   
functional inequalities, e.g.  Poincar\'e and log-Sobolev inequalities,
(\cite{MR1701522},  
\cite{MR1682772},   
\cite{MR2052235},  
\cite{MR2895086},   
and
\cite{MR3155209}). 

Such versions usually involve two arbitrary Borel measures.
A very recent result by 
  \cite{MR4061539} is not optimal yet, because it does
not contain the discrete version as given by
\cite{MR3405817}.
In Section \ref{Hiww} we shall formulate an improvement of the result by
\cite{MR4061539} that contains the discrete version by \cite{MR3405817}
as a special case. Actually our proof of this improvement is based of the discrete result of \cite{MR3405817}.
An equivalent formulation of our version of Hardy's inequality 
with weights in terms of random variables will also be given.

Furthermore, we apply our methods from Section \ref{sec:Mr} to Copson's inequality
(\cite{MR1574056}) in Section \ref{sec:Ci} and to the reverse Hardy
inequality in Section \ref{ArHi}; cf. \cite{MR854569} and \cite{MR858964}.
We treat reverse Copson inequalities in the same style in Section \ref{sec:ReverseCopsonIneq},
and we provide a probabilistic version 
of the inequalities of Carleman, P\'olya, and Knopp in Section \ref{sec:CarlemanPolyaKnopp}. 
In Section~\ref{sec:MartingalesAndHoperators} we connect our new versions of
Copson's inequality formulated in probability terms with counting process
martingales arising in survival analysis and reliability theory.
The appendix, Section~\ref{sec:Appendix}, elaborates on survival
analysis by briefly explaining connections with the forward (and backward) versions of the Kaplan - Meier
estimators appearing in right (and left) censored survival data, including a short description of the analysis of
data arising from the question of ``when do the baboons come down from the trees''.  
Other applications are presented briefly in Section \ref{ApplicRelWork} and a
summary of the new inequalities is given in Section \ref{sec:summary}.
Most of the proofs are collected in Section \ref{sec:proofs}.

\section{Hardy's inequality}\label{sec:Mr}
Here is our version of Hardy's inequality that implies both (\ref{ContHardyIneq}) and (\ref{DiscHardyIneq}).

\begin{Theorem}
{\bf Hardy's inequality}
\label{HardyIneq}\\
Let $X$ and $Y$ be independent random variables with
distribution function $F$ on $({\mathbb R}, {\cal B})$, and let $\psi $ be a nonnegative measurable function on $({\mathbb R}, {\cal B})$.
For $p>1$
\begin{equation}
\label{Hardy}
E\left( \left[ \frac{E\left( \psi(Y) {\bf 1}_{[Y \leq X]} \mid X \right)}{F(X)} \right]^p \right)
\leq \left(\frac p{p-1} \right)^p E\left( \psi^p(Y) \right) \qquad \quad
\end{equation}
holds.
For
continuous distribution functions $F$ this inequality may be rewritten as
\begin{equation}
\label{Hardyuniform}
\int_0^1 \left[ \frac 1u \int_0^u \psi_F(v) dv \right]^p du \leq \left(\frac p{p-1} \right)^p \int_0^1 \psi_F^p(v) dv
\end{equation}
with $\psi_F(v) =\psi(F^{-1}(v)),\, 0<v<1,$ and for such $F$ the constant $(p/(p-1))^p$ is the smallest possible one.
\end{Theorem}

The strength of this inequality (\ref{Hardy}) lies in the fact that it implies both the continuous
and the discrete version of Hardy's inequality.

\begin{Corollary}\label{ClassicalCont-Disc-Ineq} \hfill \newline
(i) \ For any $p>1$ and nonnegative $\psi \in L_p$,  inequality (\ref{ContHardyIneq}) holds.   \\
(ii) \ For any $p>1$ and nonnegative sequence $\{ c_n \}_{n=1}^\infty \in \ell_p$,
inequality (\ref{DiscHardyIneq}) holds.
\end{Corollary}

\begin{proof}
(i) and (ii) follow from Theorem~\ref{HardyIneq} by taking $F$ to be the distribution function corresponding to the
uniform probability measure on $[0,K]$ and on $\{ 1, \ldots , K \}$, respectively, multiplying by $K$, and taking limits as $K \rightarrow \infty$.
\end{proof}

Translating Theorem \ref{HardyIneq} from random variable notation back into analysis yields the following corollary.
\begin{Corollary}
\label{ThmAnalysisForm}
For any $p>1$, distribution function $F$ on ${\mathbb R}$, and $\psi \in L_p (F)$ we have
\begin{eqnarray*}
\int_{{\mathbb R}} | H_F \psi (x) |^p dF(x) \le \left(\frac p{p-1} \right)^p \int_{{\mathbb R}} | \psi (y) |^p dF(y)
\end{eqnarray*}
where $H_F $ is the $F-$averaging operator defined for $x \in {\mathbb R}$ and $\psi \in L_p (F)$ by
\begin{eqnarray}
H_F \psi (x) \equiv \frac{\int_{(-\infty , x]} \psi(y) d F(y) }{F(x)} = E\left( \psi(Y) \mid Y \leq x \right).
\label{HardyOperatorLeftTail}
\end{eqnarray}
\end{Corollary}

Note that $H_F$ generalizes both the discrete and the continuous Hardy averaging operators;
see e.g. \cite{MR2256532}, page 715.
Observe that $|H_F \psi| \leq H_F|\psi|$ holds for all measurable $\psi$ with equality if $\psi$ is nonnegative $F$-a.e.
This shows the equivalence of Theorem \ref{HardyIneq} and Corollary \ref{ThmAnalysisForm}.

\begin{Remark}
If we replace $({\bf 1}_{[Y \leq X]},F(X))$ in (\ref{Hardy}) by $({\bf 1}_{[Y < X]},F(X-))$
with the convention $0/0=0$, then the inequality does not hold anymore for some distribution functions with jumps.
In particular, for $X$ and $Y$ Bernoulli with success probability $P(X=1)=q$ and with $\psi(0)=1,\ \psi(1)=0$ we get
\begin{equation}\label{b}
E\left( \left[ \frac{E\left( \psi(Y) {\bf 1}_{[Y<X]} \mid X \right)}{F(X-)} \right]^p \right) = q
\end{equation}
and
\begin{equation}\label{c}
\left(\frac p{p-1} \right)^p E\left( \psi^p(Y) \right) = \left(\frac p{p-1} \right)^p (1-q).
\end{equation}
Consequently, inequality (\ref{Hardy}) with $({\bf 1}_{[Y \leq X]},F(X))$
replaced by $({\bf 1}_{[Y < X]},F(X-))$ does not hold here for
\begin{equation}\label{d}
\frac 1{1+ \left(1- \frac 1p \right)^p} < q < 1.
\end{equation}
\end{Remark}

\begin{Remark}\label{suboptimalconstant}
There are distributions for which the constant in (\ref{Hardy}) is not optimal  for any $p>1$.
This is the case for all Bernoulli distributions.
Let $X$ and $Y$ have a Bernoulli distribution with $P(X=1)=q=1-P(X=0)$.
Then with $\psi(0)=a \geq 0$ and $\psi(1)=b \geq 0$ our Hardy inequality (\ref{Hardy}) becomes
\begin{equation}\label{Bernoulli1}
(1-q) a^p + q \left( (1-q)a+ qb \right)^p \leq \left( \tfrac p{p-1} \right)^p \left( (1-q)a^p+q b^p \right).
\end{equation}
However, by convexity
\begin{eqnarray}\label{Bernoulli2}
\lefteqn{ (1-q)a^p + q \left( (1-q)a+ qb \right)^p \leq (1-q)a^p + q \left( (1-q)a^p+ q b^p \right) \nonumber } \\
&& \hspace{10em} \leq (1+q) \left( (1-q)a^p+q b^p \right)
\end{eqnarray}
holds.
Consequently, for the Bernoulli distribution with success probability $q$ the
optimal constant in our Hardy inequality equals at most $1+q$, for which
\begin{equation}\label{Bernoulli3}
1+q \leq 2 < e = \inf_{p>1} \left( 1 + \tfrac 1{p-1} \right)^p
\end{equation}
holds.
\end{Remark}

\begin{Remark}\label{othertail}
Since $-X$ has distribution function $P(X \geq -x) = 1-F_{-} (-x)$ where $F_{-} (x) \equiv F(x-)$ denotes
the left limit of $F$ at $x$,
Theorem \ref{HardyIneq} immediately implies
\begin{equation} \label{Hardyrighttail}
E\left( \left[ \frac{E\left( \psi(Y) {\bf 1}_{[Y \geq X]} \mid X \right)}{1-F(X-)} \right]^p \right)
\leq \left(\frac p{p-1} \right)^p E\left( \psi^p(Y) \right). \qquad \quad
\end{equation}
Note that (\ref{Hardyrighttail}) can be rewritten as
\begin{eqnarray*}
\int_{{\mathbb R}}  |\overline{H}_F \psi (x) |^p dF(x)
\le \left(\frac p{p-1} \right)^p \int_{{\mathbb R}} | \psi (y) |^p dF(y)
\end{eqnarray*}
where $\overline{H}_F $ is the (right-tail) $F-$averaging operator defined for $x \in {\mathbb R}$ and $\psi \in L_p (F)$ by
\begin{eqnarray}
\overline{H}_F \psi (x) \equiv \frac{\int_{[x,\infty)} \psi(y) d F(y) }{1-F(x-)}
= E\left( \psi(Y) \mid Y \geq x \right) \equiv \Psi (x) .
\label{HardyAveragingOperatorRightTail}
\end{eqnarray}
Thus
\begin{eqnarray*}
E \left( \psi(Y) - \psi(x) \mid Y \ge x \right)  = \overline{H}_F \psi (x) - \psi (x)
\end{eqnarray*}
is the ``mean residual life of $\psi (Y) $'' given $[Y\ge x]$.  In particular, with $\psi (x) = x$,
\begin{eqnarray*}
E \left( Y  - x \mid Y \ge x \right) \equiv \Psi (x) -\psi(x)
\end{eqnarray*}
is the ``mean residual life function'' corresponding to the distribution function $F$.
It turns out that for $\psi (Y) \in L_2 (F)$ and $F$ continuous
\begin{eqnarray*}
{\rm Var} (\psi(Y)) = E \left( \left\{ \psi (Y) - \Psi (Y) \right\}^2 \right),
\end{eqnarray*}
so that the conditional centering operator $I - \overline{H}_F = I - \Psi$
is an isometry.
For more on this and connections to counting process martingales and survival analysis
see
\cite{MR999013},   
\cite{MR1041385},   
and
\cite{MR1623559}.   
\cite{MR4088501}   
studies $I - H$ and $I - H^*$  as operators on $L^p ({\mathbb R}^+, \lambda )$
where $\lambda$ denotes Lebesgue measure.
\end{Remark}

\begin{Remark}\label{WellnerSaumard}
Since the conditional distribution of $X$ given $X \leq c$ has distribution function
$F(\cdot)/F(c)$ for $c \in {\mathbb R}$ and the same holds for $Y$,
we have the following conditional version of (\ref{Hardy})
\begin{eqnarray} \label{Hardyconditional}
\lefteqn{E\left( \left[ \frac{E\left( \psi(Y) {\bf 1}_{[Y \leq X]} \mid X \right)}{F(X)} \right]^p \,\middle\vert\, X \leq c \right) \nonumber } \\
&& = E\left( \left[ \frac{E\left( \psi(Y) {\bf 1}_{[Y \leq X]} \mid Y \leq c, X \right)}{F(X)/F(c)} \right]^p \,\middle\vert\,  X \leq c \right) \\
&& \leq \left(\frac p{p-1} \right)^p E\left( \psi^p(Y) \mid Y \leq c \right), \nonumber
\end{eqnarray}
where the inequality stems from (\ref{Hardy}) itself.

Similarly, we have
\begin{eqnarray} \label{Hardyconditional2}
\lefteqn{E\left( \left[ \frac{E\left( \psi(Y) {\bf 1}_{[Y \geq X]} \mid  X \right)}{1-F(X-)} \right]^p \,\middle\vert\, X > c \right) \nonumber} \\
&& = E\left( \left[ \frac{E\left( \psi(Y) {\bf 1}_{[Y \geq X]} \mid Y > c, X \right)}{(1-F(X-))/(1-F(c))} \right]^p \,\middle\vert\,  X > c \right)\\
&& \leq \left(\frac p{p-1} \right)^p E\left( \psi^p(Y) \mid Y > c \right). \nonumber
\end{eqnarray}

Together (\ref{Hardyconditional}) and (\ref{Hardyconditional2}) improve the generalization given in Theorem 3.2 of
\cite{MR3961231} 
from continuous distribution functions to arbitrary distributions, namely to
\begin{eqnarray}\label{SW}
\lefteqn{E\left( \left[ \frac{E\left( \psi(Y) {\bf 1}_{[Y \leq X]} \mid  X \right)}{F(X)} \right]^p {\bf 1}_{[X \leq c]} \right) \nonumber } \\
&& \ \ \ \ + \ E\left( \left[ \frac{E\left( \psi(Y) {\bf 1}_{[Y \geq X]} \mid  X \right)}{1-F(X-)} \right]^p {\bf 1}_{[X > c]} \right) \nonumber \\
&& = F(c) E\left( \left[ \frac{E\left( \psi(Y) {\bf 1}_{[Y \leq X]} \mid  X \right)}{F(X)} \right]^p \,\middle\vert\, X \leq c \right) \nonumber \\
&& \ \ \ \ + \ (1-F(c)) E\left( \left[ \frac{E\left( \psi(Y) {\bf 1}_{[Y \geq X]} \mid  X \right)}{1-F(X-)} \right]^p \,\middle\vert\, X > c \right) \\
&& \leq \left(\frac p{p-1} \right)^p \left[ F(c) E\left( \psi^p(Y) \mid Y \leq c \right)
                                            + (1-F(c)) E\left( \psi^p(Y) \mid Y > c \right) \right]\nonumber \\
&& = \left(\frac p{p-1} \right)^p E\left( \psi^p(Y) \right). \nonumber
\end{eqnarray}
\end{Remark}

\begin{Remark}\label{implication2}
The Hardy inequality for weighted $L_p$ spaces on $(0,\infty)$, such as Theorem 1.2.1 of
\cite{MR3408787},  
also follows from our Hardy inequality for random variables.
With $0 \le \varepsilon < (p-1)/p$ and $K$ a large constant, we choose
$F(x) = (x/K)^{1-\varepsilon p/(p-1)} \wedge 1,\ x \geq 0$.
This results in the inequality
\begin{eqnarray}\label{e}
\lefteqn{\left[1-\frac{\varepsilon p}{p-1} \right]^{p+1}
            \int_0^K  \left[\int_0^x \psi(y)y^{-\varepsilon p/(p-1)} dy \right]^p x^{p(\varepsilon -1)} dx } \\
&& \leq \left[1-\frac{\varepsilon p}{p-1} \right] \left( \frac p{p-1} \right)^p \int_0^K \psi^p(y) y^{-\varepsilon p/(p-1)} dy . \nonumber
\end{eqnarray}
Taking limits as  $K \to \infty$ and writing $\Psi(y) = \psi(y) y^{-\varepsilon p/(p-1)}$ we arrive at
\begin{equation}\label{f}
\int_0^\infty  \left[\int_0^x \Psi(y) dy \right]^p x^{p(\varepsilon -1)} dx
\leq \left[\frac {p-1}p -\varepsilon \right]^{-p} \int_0^\infty \Psi^p(y) y^{p\varepsilon} dy,
\end{equation}
which is inequality (1.2.1) combined with (1.2.3) of \cite{MR3408787}. 
Note that by choosing $\epsilon = 0$  the inequality in the last display reduces to
(\ref{ContHardyIneq}).
\end{Remark}

\section{Hardy's inequality with weights and mixed norms}\label{Hiww}
To the best of our knowledge the most recent and most general versions of
Hardy's inequalities with weights and mixed norms are
presented by \cite{MR3405817} and
 \cite{MR4061539}.
We shall improve the result of \cite{MR4061539} so that it contains the discrete version of \cite{MR3405817} as a special case.
To this end we prove the result of \cite{MR4061539} with $(-\infty,x)$ in the inner integral
replaced by $(-\infty,x]$, i.e. 
\begin{Theorem}\label{LiMao}
{\bf Hardy's Inequality with Weights and Mixed Norms}\\
Let $1<p \le q < \infty$, and suppose that $\mu$ and $\nu$ are $\sigma-$finite Borel measures on
${\mathbb R}$.  Then
\begin{eqnarray}\label{LiMao1}
\left [ \int_{{\mathbb R}} \left ( \int_{(-\infty,x]} \psi d \nu \right )^q d \mu (x) \right ]^{1/q}
\le k_{q,p}\, B \left [ \int_{{\mathbb R}} \psi^p  d \nu \right ]^{1/p}
\end{eqnarray}
holds for all measurable $\psi : {\mathbb R} \rightarrow [0,\infty)$, where $k_{q,p}$ and $B$ are defined by
\begin{equation}\label{LiMao2}
B \equiv \sup_{x \in {\mathbb R}} \mu ([x,\infty))^{1/q} \nu ((-\infty,x])^{(p-1)/p}
\end{equation}
and, with ${\rm Beta}(a,b) = \int_0^{1}  t^{a-1} (1-t)^{b-1} dt$ and $r \equiv (q-p)/p$,
\begin{equation}\label{LiMao3}
k_{q,p} \equiv \left ( \frac{r}{{\rm Beta}(1/r, (q-1)/r)} \right )^{r/q}  \ \ \mbox{and} \ \ k_{p,p} = p (p-1)^{(1-p)/p}.
\end{equation}
\end{Theorem}

\begin{Remark}\label{LiMao20}
With the help of Theorem 1.4 of \cite{MR3405817} we shall prove our Theorem \ref{LiMao} in Section \ref{sec:proofs}.
In fact these theorems are equivalent, since our theorem implies his.
For nonnegative $a_i, u_i, v_i,\ i=1, \dots,N$, let $\mu$ and $\nu$ be measures on $\{1, \dots, N \}$ that have densities $u_i$ and
$v_i^{-1/(p-1)}$, respectively, at $i$ with respect to counting measure, and let $\psi(i) = a_i v_i^{1/(p-1)},\ i=1,\dots,N$.
With these choices Theorem \ref{LiMao} yields (\ref{LiMao4}) and (\ref{LiMao5}), and hence Theorem 1.4 of \cite{MR3405817}.
\end{Remark}

\begin{Remark}\label{LiMao17}
Let $C$ be the smallest constant such that
\begin{eqnarray}\label{LiMao18}
\left [ \int_{{\mathbb R}} \left ( \int_{(-\infty,x]} \psi d \nu \right )^q d \mu (x) \right ]^{1/q}
\le C \left [ \int_{{\mathbb R}} \psi^p  d \nu \right ]^{1/p}
\end{eqnarray}
holds in the situation of Theorem \ref{LiMao}.
With $\psi(y) = {\bf 1}_{[y \leq z]}$ this yields
\begin{eqnarray}\label{LiMao19}
\lefteqn{ \left[ \nu((-\infty,z]) \mu([z,\infty)) \right]^{1/q}
\leq \left [ \int_{{\mathbb R}} \left ( \nu((-\infty,x \wedge z]) \right )^q d \mu (x) \right ]^{1/q} } \\
&& \hspace{15em} \le C \left[ \nu((-\infty,z])\right]^{1/p}, \nonumber
\end{eqnarray}
which implies the well known inequality $B \leq C$.
By Theorem~\ref{LiMao} we also have
$C \le k_{q,p}B$ so $C<\infty$ if and only if $B<\infty$.
The constants $k_{q,p}$ first appeared via a (1923) conjecture of
 \cite{MR1574995}  
which was later confirmed by
\cite{MR1574997}.    
See Chapter 5 of
\cite{MR2351524}   
for a very complete history of these developments and further results.
\end{Remark}

Theorem \ref{LiMao} and Remark \ref{LiMao17} may be reformulated in terms of random variables as follows.
\begin{Theorem}\label{LiMaorandom}
{\bf Probability Version of Hardy's Inequality with Weights and Mixed Norms}\\
Let $X$ and $Y$ be independent random variables with distribution functions $F$ and $G$
respectively, let $1<p \leq q < \infty$, and let $U$ and $V$ be nonnegative measurable functions on $({\mathbb R},{\cal B})$.
Furthermore let $C \in [0,\infty]$ be the smallest constant such that
\begin{equation}\label{LiMao25}
\left\{E\left(  \left[ E\left( \tilde{\psi}(Y) {\bf 1}_{[Y \leq X]} \mid X \right) \right]^{q} U(X) \right) \right\}^{1/q}
\leq C\, \left\{E\left( \tilde{\psi}^p(Y) V(Y) \right) \right\}^{1/p}
\end{equation}
holds for all nonnegative measurable functions $\tilde{\psi}$ on $({\mathbb R},{\cal B})$.
With
\begin{equation}\label{LiMao26}
B= \sup_{x \in {\mathbb R}} \left[\int_{[x,\infty)} U dF \right]^{1/q} \left[ \int_{(-\infty, x]} V^{-1/(p-1)} dG \right]^{(p-1)/p},
\end{equation}
the string of inequalities
\begin{equation}\label{LiMao27}
B \leq C \leq k_{q,p} B
\end{equation}
holds, even for $B=\infty$.
\end{Theorem}

\begin{proof}
Theorem \ref{LiMaorandom} is implied by Theorem \ref{LiMao} via the choices
$\mu([x,\infty)) = \int_{[x,\infty)}UdF,\\ \nu((-\infty,x]) = \int_{(-\infty,x]}V^{-1/(p-1)} dG$ and $\psi = \tilde{\psi} V^{1/(p-1)}$.

With $\mu$ a $\sigma$-finite measure and $\cup_{i=1}^\infty A_i = (0,\infty)$
a partition with $0< \mu(A_i) < \infty, i=1,2,\dots,$ the measure
$P(B) = \sum_{i=1}^\infty 2^{-i} \mu(B \cap A_i)/ \mu(A_i), B \in {\cal B},$ is a probability measure dominating $\mu$.
Let $F$ and $G$ be the distribution functions of probability measures dominating the
measures $\mu$ and $\nu$, respectively, from Theorem \ref{LiMao}.
The choices $U(x) = d\mu /dF(x)$ and $V(y) =(d\nu/dG(y))^{1-p}$ show that Theorem \ref{LiMaorandom} implies Theorem \ref{LiMao}.
\end{proof}

Following the arguments of
 \cite{MR0311856}, in Section \ref{sec:proofs}  we prove
the following generalization of his result, which is the special case $q=p$ of our Theorems \ref{LiMao} and \ref{LiMaorandom}.
\begin{Theorem}\label{MuckenhouptII}
{\bf Probability Version of Muckenhoupt's Inequality}\\
Let $X$ and $Y$ be independent random variables with distribution functions $F$ and $G$
respectively, let $p>1$, and let $U$ and $V$ be nonnegative measurable functions on $({\mathbb R},{\cal B})$.
Furthermore let $C \in [0,\infty]$ be the smallest constant such that
\begin{equation}\label{M1II}
E\left(  \left[ E\left( \psi(Y) {\bf 1}_{[Y \leq X]} \mid X \right) \right]^{p} U(X) \right)
\leq C\, E\left( \psi^p(Y) V(Y) \right)
\end{equation}
holds for all nonnegative measurable functions $\psi$ on $({\mathbb R},{\cal B})$.
With
\begin{equation}\label{M2II}
B= \sup_{x \in {\mathbb R}} \int_{[x,\infty)} U dF \left[ \int_{(-\infty, x]} V^{-1/(p-1)} dG \right]^{p-1} ,
\end{equation}
the string of inequalities
\begin{equation}\label{M3.5II}
B \leq C \leq \frac {p^p}{(p-1)^{p-1}} B
\end{equation}
holds, even for $B=\infty$.
\end{Theorem}

\begin{Remark}
With $U= G^{-p}, V=1$ and $G=F$ the second inequality in (\ref{M3.5II}) does {\em not}
imply our Hardy inequality (\ref{Hardy}).
Indeed, for Bernoulli random variables with $P(X=1)= 1/p = 1-P(X=0)$ the factor $B$
equals $1+(p-1)^{p-1}/p^p$ then and hence the upper bound on $C$ equals $1+ p^p/(p-1)^{p-1}$,
which is larger than $(p/(p-1))^p$ for $p \geq p_0 \approx 1.77074$.  \\
However, with $U= G^{-p}, V=1$ and $G=F$ a continuous distribution function the factor
$B$ equals $1/(p-1)$, which shows that (\ref{M3.5II}) {\em does} imply our Hardy inequality (\ref{Hardy}) for this case.
\end{Remark}

If $X$ is stochastically larger than $Y,\ Y \preceq X,$ and they have no point masses at the same location,
then Theorem \ref{MuckenhouptII} yields an inequality very similar to (\ref{Hardy}). A comparable result is obtained for $X \preceq Y$.
\begin{Corollary}
\label{HardyOrdering1}
{\bf Stochastic ordering}\\
Let $X$ and $Y$ be independent random variables with distribution functions $F$ and $G$
respectively, let $p>1$, and let $\psi$ be a nonnegative measurable function on $({\mathbb R},{\cal B})$. \\
(a) If $P(X=Y)=0$ and $F(x) \le G(x),\, x \in {\mathbb R},$ hold, then
\begin{equation}
\label{HardyOrdering2}
E\left( \left[ \frac{E\left( \psi(Y) {\bf 1}_{[Y \leq X]} \mid X \right)}{G(X)} \right]^p \right)
\leq \left(\frac p{p-1} \right)^p E\left( \psi^p(Y) \right) \qquad \quad
\end{equation}
is valid. \\
(b) If $F$ is continuous and $F(x) \ge G(x),\, x \in {\mathbb R},$ holds, then
\begin{equation}
\label{HardyOrdering3}
E\left( \left[ \frac{E\left( \psi(Y) {\bf 1}_{[Y \leq X]} \mid X \right)}{F(X)} \right]^p \right)
\leq \left(\frac p{p-1} \right)^p E\left( \psi^p(Y) \right) \qquad \quad
\end{equation}
is valid.
\end{Corollary}
\begin{proof}
In case (a) we apply Theorem \ref{MuckenhouptII} with $U= G^{-p}$ and $V=1$.
Then $B$ from (\ref{M2II}) equals
\begin{equation}\label{HardyOrdering4}
B= \sup_{r \in {\mathbb R}}  G^{p-1}(r) \int_{[r,\infty)} G^{-p} dF.
\end{equation}
If $F$ has no point mass at $r$, then the stochastic ordering $Y \preceq X$ implies
\begin{eqnarray}\label{HardyOrdering5}
\lefteqn{ \int_{[r,\infty)} G^{-p} dF = \int_{(r,\infty)} G^{-p} dF \leq \int_{(r,\infty)} G^{-p} dG \nonumber } \\
&& = \int_{G^{-1}(u) > r} (G(G^{-1}(u)))^{-p} du \leq \int_{[G(r),1)} (G(G^{-1}(u)))^{-p} du \\
&& \leq \int_{[G(r),1)} u^{-p} du = \frac 1{p-1} \left[ G^{1-p}(r) -1 \right] \leq \frac{G^{1-p}(r)}{p-1}. \nonumber
\end{eqnarray}
In the first line of the last display and in the second line below
 we use the characterization
$Y \preceq X$ if and only if $E h(Y) \le E h(X)$ for all bounded and non-decreasing functions $h$;
see e.g.
\cite{MR1889865}  
Theorem 1.2.8 (ii),  page 5, or
\cite{MR2265633}    
(1.A.7), page 4.

If $F$ has a point mass at $r$ then $G$ has not and the stochastic ordering $Y \preceq X$ implies
\begin{eqnarray}\label{HardyOrdering6}
\lefteqn{ \int_{[r,\infty)} G^{-p} dF \leq \int_{[r,\infty)} G^{-p} dG } \\
&& = \int_{(r,\infty)} G^{-p} dG = \frac 1{p-1} \left[ G^{1-p}(r) -1 \right] \leq \frac{G^{1-p}(r)}{p-1}. \nonumber
\end{eqnarray}
Combining (\ref{HardyOrdering4})--(\ref{HardyOrdering6}) and (\ref{M1II})--(\ref{M3.5II}) we arrive at (\ref{HardyOrdering2}).\\
In case (b) we apply Theorem \ref{MuckenhouptII} with $U= F^{-p}$ and $V=1$.
Then the continuity of $F$ and $G \leq F$ imply that $B$ from (\ref{M2II}) satisfies
\begin{eqnarray}\label{HardyOrdering7}
\lefteqn{ B = \sup_{r \in {\mathbb R}} G^{p-1}(r) \left (\int_{[r,\infty)} F^{-p} dF  \right ) 
            = \sup_{r \in {\mathbb R}} \frac 1{p-1} \left[F^{1-p}(r) -1 \right] G^{p-1}(r) \nonumber } \\
&& \hspace{12em} \leq \sup_{r \in {\mathbb R}} \frac 1{p-1} \left[ \frac{G(r)}{F(r)} \right]^{p-1} = \frac 1{p-1}
\end{eqnarray}
and hence that (\ref{HardyOrdering3}) holds.
\end{proof}

\section{A reverse Hardy inequality}\label{ArHi}

There are also reversed versions of the classical Hardy inequality:
the {\sl continuous (or integral form) inequality} says,
if $p>1$ and $\psi$ is a nonnegative, nonincreasing $p-$integrable function on $(0,\infty)$, then
\begin{eqnarray}\label{ReverseContHardyIneq}
\int_0^\infty \left ( \frac{1}{x} \int_0^x \psi (y) dy \right )^p \, dx
\geq \frac{p}{p-1} \int_0^{\infty} \psi^p (y) dy,
\end{eqnarray}
while the {\sl discrete  (or series form) inequality} says,
if $p>1$ and  $\{ c_n \}_1^\infty$ is a nonincreasing sequence of nonnegative real numbers,
then
\begin{eqnarray}\label{ReverseDiscHardyIneq}
\sum_{n=1}^\infty \left ( \frac{1}{n} \sum_{k=1}^n c_k \right )^p
\geq \zeta(p)  \sum_{k=1}^\infty c_{k}^p .
\end{eqnarray}
Here, $\zeta(\cdot)$ is the zeta function.
These inequalities have been obtained independently by \cite{MR854569} 
and
 \cite{MR858964};  
see also Lemma 2.1 of \cite{MR1438149}.  
By taking $\psi$ the indicator function of the unit interval we see that (\ref{ReverseContHardyIneq}) is sharp and by taking
$c_1=1, c_2= c_3 = \cdots =0$ that (\ref{ReverseDiscHardyIneq}) is sharp.

Here are our random variable versions of (\ref{ReverseContHardyIneq}) and (\ref{ReverseDiscHardyIneq}).

\begin{Theorem}
{\bf Reverse Hardy inequality}
\label{ReverseHardyIneq}\\
Let $X$ and $Y$ be independent random variables both with distribution function
$F$ on $({\mathbb R}, {\cal B})$, and let $\psi $ be a nonnegative,
nonincreasing measurable function on $({\mathbb R}, {\cal B})$.
For $p>1$ and $F$ absolutely continuous
\begin{eqnarray}
E\left( \left[ \frac{E\left( \psi(Y) {\bf 1}_{[Y \leq X]} \mid X \right)}{F(X)} \right]^p \right)
& \ge & \frac p{p-1} E\left( \psi^p(Y) \left[1-F^{p-1}(Y) \right]  \right) \nonumber  \\
& \ge & E\left( \psi^p(Y) \right) \label{ReverseDensityHardy}
\end{eqnarray}
holds with equalities if $\psi$ is constant.

For $p \geq 1$ and $F$ general
\begin{equation}\label{ReverseGeneralHardy}
E\left( \left[ \frac{E\left( \psi(Y) {\bf 1}_{[Y \leq X]} \mid X \right)}{F(X)} \right]^p \right) \geq E\left( \psi^p(Y) \right)
\end{equation}
holds with equalities if $\psi$ is constant.

If $F$ is general, but $p \geq 2$ is an integer, then, with
$X,Y,X_1, \dots, X_p$ independent and identically distributed
and with $X_{(p)}=\max \{X_1, \dots, X_p \}$, we have
\begin{equation}\label{ReverseIntegerHardy}
E\left( \left[ \frac{E\left( \psi(Y) {\bf 1}_{[Y \leq X]} \mid X \right)}{F(X)} \right]^p \right)
\geq E\left( \psi^p(X_{(p)}) E \left( F^{-p}(Y) {\bf 1}_{[Y \geq X_{(p)}]} \mid X_{(p)} \right) \right)
\end{equation}
with equality if $\psi$ is constant.
\end{Theorem}
\medskip

The continuous version (\ref{ReverseContHardyIneq}) of the reverse Hardy inequality is
contained in (\ref{ReverseDensityHardy}) and the discrete version (\ref{ReverseDiscHardyIneq})
for integer $p$ follows from (\ref{ReverseIntegerHardy}).

\begin{Corollary}\label{ReverseInequalities} \hfill \newline
(i) \ For any $p>1$ and nonnegative, nonincreasing $\psi \in L_p$,  inequality (\ref{ReverseContHardyIneq}) holds.   \\
(ii) \ For any integer $p>1$ and nonnegative, nonincreasing sequence $\{ c_n \}_{n=1}^\infty \in \ell_p$,
inequality (\ref{ReverseDiscHardyIneq}) holds.
\end{Corollary}

For further developments concerning reverse Hardy type inequalities, see
\cite{MR2376257}. 

\section{Copson's inequality}\label{sec:Ci}

 \cite{MR1574056} 
presented the following  pair of inequalities:
the {\sl continuous (or integral form) inequality} says,
if $p>1$ and $\psi$ is a nonnegative $p-$integrable function on $(0,\infty)$, then
\begin{equation}\label{ContCopsonIneq}
\int_0^\infty \left ( \int_x^\infty \frac{\psi (y)}{y} dy \right )^p dx
\leq p^p \int_0^{\infty} \psi^p(y) dy
\end{equation}
holds, while the {\sl discrete  (or series form) inequality} says,
if $p>1$ and $a_i$ and $\lambda_i,\, i=1,2,\dots,$ are nonnegative
numbers and $\Lambda_i = \sum_{j=1}^i \lambda_j,\, i=1,2,\dots,$ is positive, then
\begin{equation}\label{DiscCopsonIneq}
\sum_{i=1}^\infty \left[ \sum_{j=i}^\infty a_j \frac {\lambda_j}{\Lambda_j} \right]^p \lambda_i
\leq p^p \sum_{j=1}^\infty a_{j}^p \lambda_{j}
\end{equation}
holds.
We generalize Copson's inequalities as follows.

\begin{Theorem}
{\bf Copson's inequality}
\label{NCopsonIneq}\\
Let $X$ and $Y$ be independent random variables with distribution function
$F$ on $({\mathbb R}, {\cal B})$, and let $\psi $ be a nonnegative
measurable function on $({\mathbb R}, {\cal B})$.
For $p \ge 1$
\begin{equation}
\label{NCopson}
E\left( \left[ E \left( \frac{\psi(Y)}{F(Y)} {\bf 1}_{[Y \geq X]} \mid X \right) \right]^p \right)
\leq p^p E\left( \psi^p(Y) \right)
\end{equation}
holds.
For absolutely continuous distribution functions $F$ the constant $p^p$ is the smallest possible one.
\end{Theorem}

The strength of this inequality (\ref{NCopson}) lies in the fact that it implies both the continuous
and the discrete version of Copson's inequality.

\begin{Corollary}\label{CopsonCont-Disc-Ineq} \hfill \newline
(i) \ For any $p \ge 1$ and nonnegative $\psi \in L_p$,  inequality (\ref{ContCopsonIneq}) holds.   \\
(ii) \ For any $p \ge 1$ and nonnegative sequences
$\{ a_n \}_{n=1}^\infty, \{ \lambda_n \}_{n=1}^\infty \in \ell_p$ with $\lambda_1 > 0$,
inequality (\ref{DiscCopsonIneq}) holds.
\end{Corollary}

\begin{proof}
By Tonelli's theorem (Fubini) equality holds in (\ref{ContCopsonIneq}) and (\ref{DiscCopsonIneq}) for $p=1$.
Let $p>1$. \\
(i) can be seen by choosing $X$ and $Y$ uniform on $(0,K)$ and taking limits with $K \to \infty$. \\
(ii) needs a longer argument.
For $p>1$ define $\Lambda_i=\sum_{j=1}^i \lambda_j,\, p_i = \lambda_i/\Lambda_K,\, i=1,\dots,K$,
for some natural number $K$ and define the bounded continuous function $\psi$ such that
$\psi(i)=a_i$ holds for $i=1,\dots,K$.
With $F(x) = \sum_{i=1}^{K \wedge \lfloor x \rfloor} p_i$ Theorem \ref{NCopsonIneq} yields
\begin{eqnarray}\label{Copson10}
\lefteqn{ E\left( \left[ E \left( \frac{\psi(Y)}{F(Y)} {\bf 1}_{[Y \geq X]} \mid X \right) \right]^p \right)
= \sum_{i=1}^K \left[ \sum_{j=i}^K \frac{a_j}{\Lambda_j /\Lambda_K} p_j \right]^p p_i } \\
&& = \sum_{i=1}^K  \left[ \sum_{j=i}^K a_j \frac {\lambda_j}{\Lambda_j} \right]^p \frac {\lambda_i}{\Lambda_K}
\leq p^p \sum_{j=1}^K a_j ^p \frac {\lambda_j }{\Lambda_K} =  p^p E\left( \psi^p(Y) \right). \nonumber
\end{eqnarray}
For $K_1 \leq K_2$ this implies
\begin{equation}\label{Copson11}
\sum_{i=1}^{K_1} \left[ \sum_{j=i}^{K_2} a_j \frac {\lambda_j}{\Lambda_j}\right]^p \lambda_i
\leq p^p \sum_{j=1}^{K_2} a_j^p \lambda_j.
\end{equation}
Taking limits here for $K_2 \to \infty$ and subsequently $K_1 \to \infty$ we arrive at (\ref{DiscCopsonIneq}).
\end{proof}

Comparison of the left side of (\ref{NCopson}) with the left side of (\ref{Hardy}) and the definition of
$H_F$ in (\ref{HardyOperatorLeftTail}) leads us to define the Copson (or dual) operator $H_F^*$
as follows:   for $x \in {\mathbb R}$ and $\psi \in L_p (F)$
\begin{eqnarray}
H_F^* \psi (x) \equiv \int_{[x,\infty)} \frac{\psi (y)}{F(y)} d F(y) = \int_{[x,\infty)} \psi (y) d \Lambda (y)
\label{DualHardyOperator}
\end{eqnarray}
where $\Lambda(x) \equiv \int_{[x, \infty )}  dF(y) / F(y)$  is the {\sl reverse (or backward)} hazard function
corresponding to $F$.  (We will introduce and discuss the forward hazard function
$\overline{\Lambda} (x) \equiv  \int_{(-\infty, x]} d F (x) / (1-F(x-))$ in connection with
the inequalities of Carleman, P\'olya, and Knopp in Section~\ref{sec:CarlemanPolyaKnopp}.)

As pointed out by  Hardy in \cite{MR1574154}, the discrete Copson inequality is a ``reciprocal'' or ``dual'' inequality
of the discrete Hardy inequality (\ref{DiscHardyIneq}), in the sense that
one implies the other.
But this holds in other senses as well.
For a treatment of (\ref{ContHardyIneq}) and (\ref{ContCopsonIneq})
based on the duality of $L_p$
and $L_q$ with $1/p + 1/q =1$,
see \cite{MR1681462}, section 6.3, especially his Theorem 6.20 and Corollary 6.2.1.
In particular when viewed as operators on $L_2 (F)$,
$H_F$ and $H_F^*$ are adjoint operators:
for $\psi$ and $\chi$ in $L_2(F)$ we have
\begin{eqnarray}
\lefteqn{ E\left( \left[ \frac{E\left( \psi(Y) {\bf 1}_{[Y \leq X]} \mid X \right)}{F(X)} \right] \chi(X) \right)
            = E\left( \frac{\psi(Y) \chi(X)}{F(X)} {\bf 1}_{[Y \leq X]} \right) \nonumber } \\
&& \hspace{10em} = E\left( \left[ E\left( \frac{\chi(X)}{F(X)} {\bf 1}_{[Y \leq X]} \mid Y \right) \right] \psi(Y) \right).
\end{eqnarray}
So, $H_F$ and $H_F^*$ have the same norms for $p=2$, and indeed
the bounds in (\ref{summaryHardy})
and (\ref{summaryCopson}) are the same for $p=2$.
Applying Hardy's approach we obtain the equivalence of (\ref{Hardy}) and (\ref{NCopson}).

\begin{Theorem}\label{equivalence}
{\bf Equivalence of Hardy's and Copson's inequality} \\
Let $X$ and $Y$ be independent random variables with distribution function $F$ on $({\mathbb R}, {\cal B})$.
For $p>1$ and all nonnegative measurable functions $\psi$ on $({\mathbb R}, {\cal B})$ (\ref{Hardy}) holds if and only if
for $p>1$ and all nonnegative measurable functions $\psi$ on $({\mathbb R}, {\cal B})$ (\ref{NCopson}) holds.
\end{Theorem}

Although this Theorem~ \ref{equivalence}
(formally) renders one of our proofs of Hardy's and Copson's inequality
superfluous, we have included both proofs in Section \ref{sec:proofs} to illustrate the different methods.

\begin{Remark}\label{suboptimalconstantCopson}
 For $p>1$ there are distributions for which the constant $p^p$ in (\ref{NCopson}) is not optimal.
This is the case for all Bernoulli distributions.
Let $X$ and $Y$ have a Bernoulli distribution with $P(X=1)=q=1-P(X=0)$.
Then with $\psi(0)=a$ and $\psi(1)=b$ the left hand side of our Copson inequality (\ref{NCopson}) equals
\begin{eqnarray}\label{Bernoulli7}
\lefteqn{ (1-q)(a + qb)^p + q (qb)^p = (1-q)(1+q)^p\left(\frac 1{1+q}a + \frac q{1+q}b \right)^p + q^{p+1}b^p \nonumber } \\
&&\leq (1-q)(1+q)^{p-1} \left( a^p + q b^p \right) + q^{p+1}b^p \\
&& = (1+q)^{p-1} \left( (1-q) a^p + q b^p \right) + q^2 \left( q^{p-1} - (1+q)^{p-1} \right) b^p \hspace{5em} \nonumber \\
&& \leq (1+q)^{p-1} \left( (1-q) a^p + q b^p \right), \nonumber
\end{eqnarray}
where the first inequality follows from Jensen's inequality and the convexity of $x \mapsto x^p,\, x \geq 0.$
The right hand side of (\ref{Bernoulli7}) is bounded by
\begin{equation}\label{Bernoulli8}
2^{p-1} \left( (1-q) a^p + q b^p \right) < p^p \left( (1-q) a^p + q b^p \right),
\end{equation}
where the strict inequality holds since $p \mapsto p\log p - (p-1) \log 2$
is strictly increasing on $[1,\infty)$ with value 0 at $p=1$
and where the last expression is the upper bound in (\ref{NCopson}).
\end{Remark}

\par\noindent
\begin{Remark}
Theorem 7 gives a qualitative connection between Hardy's inequality and Copson's
inequality (or the ``dual Hardy inequality"). The papers by
\cite{MR2390520},  
\cite{MR3180926},     
and
\cite{MR4051866}      
quantify these connections.  These results are strongly related to further work on the
connections between the $I-H_F$ and $I - H_F^*$ operators on the one hand, and between
the $I - \overline{H}_F$ and $I - \overline{H}_F^*$ operators on the other hand.
Also see \cite{MR2747011}. 
Recall that
 \begin{eqnarray*}
 H_F \psi (x) & \equiv & \frac{ \int_{(-\infty,x]} \psi (y) dF(y)}{F(x)} , \ \ \
 \overline{H} _F \psi (x)  \equiv  \frac{ \int_{[x,\infty)} \psi (y) dF(y)}{1-F(x-)} , \\
 H_F^* \psi (x) & \equiv &  \int_{[x,\infty)} \frac{\psi (y)}{F(y)} d F(y)    \ \ \  
\ \ \, \overline{H}_F^* \psi (x)\, \equiv  \int_{(-\infty, x]} \frac{\psi (y)}{1-F(y-)} d F(y)  \\
& = & \int_{[x,\infty)} \psi (y) d \Lambda (y),  \qquad \qquad \ \ \ \ \ \  = \int_{(-\infty ,x]} \psi (y) d \overline{\Lambda} (y) ,
\end{eqnarray*}
where
\begin{eqnarray}
\Lambda(x) \equiv \int_{[x,\infty)} \frac{dF(y)}{F(y)} , \qquad \qquad
\overline{ \Lambda}  (x) \equiv \int_{(-\infty, x]} \frac{1}{1-F(y-)} dF(y) .
\label{CumulativeHazardFunctions}
\end{eqnarray}
are  the {\sl backward} cumulative hazard function  and the ({\sl forward}) cumulative hazard
functions of survival analysis.
\end{Remark}

\section{A reverse Copson inequality}
\label{sec:ReverseCopsonIneq}

Reversed versions of the classical Copson inequality are given in Theorems 2 and 4 of Renaud (1986) \cite{MR854569}.
His continuous (or integral form) inequality may be rephrased as follows.
If $p \geq 1$ holds and $\psi$ is a nonnegative $p-$integrable function on $(0,\infty)$ such that $x \mapsto \psi(x)/x$ is nonincreasing, then
\begin{equation}\label{RenaudRevContCopsonIneq}
\int_0^\infty \left ( \int_x^\infty \frac{\psi (y)}{y} dy \right )^p dx \geq \int_0^{\infty} \psi^p(y) dy
\end{equation}
holds.  His discrete form says:  if  $p \geq 1$ holds and $a_1/1 \geq a_2/2 \geq \cdots$ are nonnegative numbers, then
\begin{equation}\label{RenaudRevDiscCopsonIneq}
\sum_{i=1}^\infty \left[ \sum_{j=i}^\infty \frac {a_j}j \right]^p \geq \sum_{i=1}^\infty a_i^p
\end{equation}
holds.

It seems natural to consider a reverse Copson inequality formulated in terms of random variables.
Here is our result in this direction.

\begin{Theorem}
{\bf Reverse Copson inequality}
\label{NewReverseCopsonIneq}\\
Let $X$ and $Y$ be independent random variables both with distribution function $F$ on $({\mathbb R}, {\cal B})$
and let $\psi $ be a nonnegative $p$-integrable function on $({\mathbb R}, {\cal B})$ with $p \in [1, \infty)$.
If the distribution function $F$ is continuous and $x \mapsto \psi(x)/F(x)$ is nonincreasing, then
\begin{equation}\label{NRC0}
E\left( \left[ E \left( \frac{\psi(Y)}{F(Y)} {\bf 1}_{[Y \geq X]} \mid X \right) \right]^p \right)
\geq E\left( \psi^p(Y) \right)
\end{equation}
holds with equality if $\psi = F$ or $p=1$ holds.

If the distribution function $F$ is continuous, $\psi$ is nonincreasing, and $p$ is an integer, then
\begin{equation}\label{NRC1}
E\left( \left[ E \left( \frac{\psi(Y)}{F(Y)} {\bf 1}_{[Y \geq X]} \mid X \right) \right]^p \right)
\geq p!\, E\left( \psi^p(Y) \right)
\end{equation}
holds with equality if $\psi$ is constant, F is degenerate, or $p=1$ holds.

If the distribution function $F$ is arbitrary, $\psi$ is nonincreasing, and $p$ is an integer, then
\begin{equation}\label{NRC2}
E\left( \left[ E \left( \frac{\psi(Y)}{F(Y)} {\bf 1}_{[Y \geq X]} \mid X \right) \right]^p \right)
\geq E\left( \psi^p(Y) \right)
\end{equation}
holds with equality if $\psi$ equals $0$, or $F$ is degenerate, 
or $p=1$ holds.
\end{Theorem}

We conjecture that (\ref{NRC1}), with $p!\,$ replaced by $\Gamma(p+1)$, and (\ref{NRC2})
hold for all $p \geq 1$, but we have no proof.
Note that for $F$ continuous (\ref{NRC2}) with $p \in [1,\infty)$ follows from (\ref{NRC0}).
For the situations of the continuous and discrete versions of the original Copson inequality
our reverse Copson inequality implies:

\begin{Corollary}
\label{ReverseCopson-continuous-discrete}
$\phantom{bb}$\\
(i) With $p \in [1,\infty)$ and $\psi$ nonnegative $p-$integrable on $(0,\infty)$ such that $x \mapsto \psi(x)/x$
is nonincreasing (\ref{RenaudRevContCopsonIneq}) holds. \\
(ii)   If $p \geq 1$ is an integer and $\psi$ is a nonnegative, nonincreasing, $p-$integrable function on $(0,\infty)$, then
\begin{equation}\label{RevContCopsonIneq}
\int_0^\infty \left ( \int_x^\infty \frac{\psi (y)}{y} dy \right )^p dx
\geq p! \int_0^{\infty} \psi^p(y) dy
\end{equation}
holds.  \\
(iii)  If $p \geq 1$ is an integer and $a_1 \geq a_2 \geq \cdots$ and $\lambda_i,\, i=1,2,\dots,$
are nonnegative numbers and $\Lambda_i = \sum_{j=1}^i \lambda_j,\, i=1,2,\dots,$ is positive, then
\begin{equation}\label{RevDiscCopsonIneq}
\sum_{i=1}^\infty \left[ \sum_{j=i}^\infty a_j \frac {\lambda_j}{\Lambda_j}\right]^p \lambda_i
\geq \sum_{j=1}^\infty a_j^p \lambda_j
\end{equation}
holds.
\end{Corollary}
The proof of this corollary is almost the same as the proof of
Corollary \ref{CopsonCont-Disc-Ineq} in Section \ref{sec:Ci}
(but with
the inequality signs reversed and the constants changed), and therefore it is omitted.

\begin{Remark}\label{NRC5}
Without continuity of $F$ inequality (\ref{NRC0}) is not generally valid.
Again a counterexample is provided by the Bernoulli distribution.
Take $\psi = F$ and $F(x) = (1-q) {\bf 1}_{[x \geq 0]} + q {\bf 1}_{[x \geq 1]}$.
Now, as a function of the success probability $q$ the left minus the right hand side of (\ref{NRC0}) equals
\begin{equation}\label{NRC6}
E\left([1-F(X-)]^p \right) - E \left([F(X)]^p \right) = (1-2q) \left[1-(1-q)^p \right],
\end{equation}
which is negative for $1/2 < q \leq 1$.
\end{Remark}

\section{The Carleman and P\'olya - Knopp inequalities}
\label{sec:CarlemanPolyaKnopp}
Another classical pair of inequalities in this family of inequalities are those associated with the
names of P\'olya and Knopp in the continuous (or integral) case, and Carleman in the discrete case:
for a positive function $\psi $ in $L_1 ({\mathbb R}^+ , \lambda)$,
\begin{eqnarray}
\int_0^\infty \exp \left ( \frac{1}{x} \int_0^x \log \psi (y) dy \right ) dx \le e \cdot \int_0^{\infty} \psi (y) dy
\label{ContinuousCarlemanPolya-Knopp}
\end{eqnarray}
and, for a sequence of constants $\{ c_k \}$,
\begin{eqnarray}
\sum_{k=1}^\infty \left ( \prod_{j=1}^k c_j \right )^{1/k}  \le   e \cdot \sum_{j=1}^\infty  c_j ;
\label{DiscreteCarlemanPolya-Knopp}
\end{eqnarray}
see e.g.
\cite{MR2256532}   
section 9,
\cite{MR2148234}   
and
\cite{MR1820809}.   
By now the reader will anticipate our impulse to reformulate and unify these two inequalities
in a more probabilistic vein involving random variables and distribution functions as follows:

\begin{Theorem} \label{CarlesonPolya-Knopp}
Let $\psi$ be a positive valued function on ${\mathbb R}$ and let $X,Y $ be independent random
variables with distribution function $F$.   If $\psi \in L_1 (F)$ then
\begin{eqnarray*}
E \left \{ \exp \left (   \frac{ E \left (  1_{[Y \le X ]} \log \psi (Y) | X \right ) }{F(X)} \right ) \right \}
\le e \cdot E \psi (Y) .
\end{eqnarray*}
\end{Theorem}

\begin{Corollary}\label{Carleman-Polya-KnoppCor} \hfill \newline
(i) \ For any  nonnegative $\psi \in L_1$,  inequality (\ref{ContinuousCarlemanPolya-Knopp})  holds.\\
(ii) \ For any positive sequence $\{ c_k \} \in \ell_1$ the inequality (\ref{DiscreteCarlemanPolya-Knopp})
holds
\end{Corollary}

The proof of Corollary \ref{ClassicalCont-Disc-Ineq} is applicable to Corollary \ref{Carleman-Polya-KnoppCor} as well.
\medskip

 \cite{MR1920123}    
rewrite the classical integral version
of the Carleman inequality as follows:
replacing $\psi(y)$ in (\ref{ContinuousCarlemanPolya-Knopp}) by $\psi (y) /y$ yields
\begin{eqnarray}
\int_0^{\infty} \exp \left ( \frac{1}{x} \int_0^x \log \psi (y) dy \right ) \frac{dx}{x}   <  \int_0^\infty \psi (x) \frac{dx}{x}.
\label{KPOrewriteOfCarlemanIneq}
\end{eqnarray}
This follows by elementary manipulations together with the identity $\int_0^x \log y dy = x (\log x -1)$.
\cite{MR1920123} give an alternative proof of (\ref{ContinuousCarlemanPolya-Knopp})
by proving (\ref{KPOrewriteOfCarlemanIneq})
via the following simple convexity argument.
By convexity of {\rm exp},  it follows from Jensen's inequality followed by Fubini's theorem that
\begin{eqnarray*}
\int_0^{\infty} \exp \left ( \frac{1}{x} \int_0^x \log \psi (y) dy \right ) \frac{dx}{x}
&\le & \int_0^{\infty}  \frac{1}{x^2} \left \{  \int_0^x \psi (y) dy \right \} dx  \\
& = & \int_0^\infty \psi (y) \left \{ \int_y^\infty \frac{1}{x^2}  dx \right \}  dy = \int_0^{\infty} \psi (y) \frac{dy}{y} .
\end{eqnarray*}
Strict inequality follows because equality in Jensen's inequality almost everywhere forces $\psi$
to be constant a.e., but this contradicts finiteness of $\int_0^\infty \psi (y)/y \ dy$.

Now several questions arise:  is there a corresponding rewrite of our probabilistic version of the inequalities
of Carleman and P\'olya - Knopp?   The answer is clearly ``yes'' for continuous distribution functions $F$.
Replacing $\psi $ by $\psi /F$ in (\ref{CarlesonPolya-Knopp}) and arguing as above, but using the identity
$\int_{(-\infty,x]} \log F(y) dF(y) = F(x) ( \log F(x) - 1)$, yields
\begin{eqnarray*}
\int_{{\mathbb R}} \exp \left ( \frac{1}{F(x)}  \int_{(-\infty ,x]}
\log \psi(y)  d F(y) \right )  \frac{d F(x)}{F(x) } 
& < & \int_{{\mathbb R}} \psi (y) \frac{dF(y)}{F(y)} \\
& = & \int_{{\mathbb R}} \psi (y) d (- \Lambda (y))
\end{eqnarray*}
where  $\Lambda (x) \equiv   \int_{[x,\infty)} dF(y)/F(y)$.
This is a ``left tail inequality'' with motivations from survival analysis.

For the corresponding ``right tail inequality'' we instead replace $\psi$ by $\psi /(1-F)$.
Then reasoning as above yields, for continuous $F$,
\begin{eqnarray*}
\int_{{\mathbb R}} \exp \left ( \frac{1}{1- F(x-)} \int_{[x,\infty)}  \log \psi (y) d F(y) \right )d\overline{\Lambda} (x)
\le \int_{{\mathbb R}} \psi (y) d \overline{\Lambda} (y)
\end{eqnarray*}
where
$\overline{\Lambda} (x) \equiv \int_{(-\infty, x]}  dF(y)/(1-F(y-))$.\\
{\bf Note:}  This notation goes against the classical notation of survival analysis but is in
keeping with the current notation of our paper.   The usual notation for the
``right side'' or forward cumulative hazard function is simply
$\Lambda (x) = \int_{(-\infty,x]}  d F(y) / (1-F(y-))$.

\section{Martingale connections and the $H$ operators}
\label{sec:MartingalesAndHoperators}

In this section we expand on the comments in Sections \ref{sec:Mr}, \ref{sec:Ci}, and  \ref{sec:CarlemanPolyaKnopp} concerning
martingales, counting processes, and the residual life and dual Hardy operators.

First recall the operators $H_F$,  $\overline{H}_F$, $H_F^*$ and $\overline{H}_F^*$  introduced in Section 5.
With $I$ the identity operator and $F$ the {\em continuous} distribution function of $X$,   Fubini's theorem yields
\begin{equation}\label{MC1}
(I-H_F)(I-H_F^*) \psi = \psi, \quad (I-H_F^*)(I-H_F) \psi = \psi - E\psi(X).
\end{equation}
We will also need the classical Hardy operators $H$ and $H^*$ defined by
\begin{eqnarray*}
H \psi (x) \equiv \frac{1}{x} \int_0^x \psi (y) dy, \ \ \ \mbox{and}  \ \ \ H^* \psi (x) \equiv \int_x^\infty \frac{\psi (y)}{y}  dy .
\end{eqnarray*}
for $\psi \in L_p ({\mathbb R}_+ , \lambda)$ where $\lambda$ denotes Lebesgue measure.
 \cite{MR1676324}  
(see also  \cite{MR2390520}),
showed that
\begin{eqnarray}
(H-I)^{-1} \psi (x) = H^* \psi (x) - \psi (x) = \int_x^{\infty} \frac{\psi(y)}{y} dy  - \psi (x) .
\label{ClassicalInverseOperator}
\end{eqnarray}
It is well known (see e.g.
\cite{MR187085})
that $I-H$ is an isometry on $L_2 ({\mathbb R}_+)$.

\cite{MR999013} 
showed that $R\equiv I - H_F$  is an isometry of $L_2 ({\mathbb R}_+ , F)$;
see also
\cite{MR1623559}   
Appendix A.1, pages 420 - 424.
These authors also
showed that with $R \equiv I - \overline{H}_F$ and
$L\equiv I - \overline{H}_F^* $ we have
\begin{eqnarray*}
R \circ L \psi = \psi  \ \ \mbox{and } \ \
L \circ R \psi = \psi - E_F \psi (X)
\end{eqnarray*}
for $\psi \in L_2 (F)$.
Thus $R^{-1} = L$ on $L_2^0 (F) \equiv \{ \psi \in L_2 (F) : \ E_F \psi (X) = 0 \}$,
and we see that the analogue of the identity (\ref{ClassicalInverseOperator}) becomes
\begin{eqnarray}
L \psi (x)
& = & R^{-1} \psi (x) = (\overline{H}_F - I)^{-1} \psi (x)   \nonumber \\
      & = & \overline{H}_F^* \psi (x) - \psi (x)  \nonumber  \\
      & = & - \left ( \psi (x) - \int_{(-\infty,x]} \psi (y) d \overline{\Lambda} (y) \right )
\label{L-Operator-Identity}
\end{eqnarray}
where $\overline{\Lambda}$ is as defined in (\ref{CumulativeHazardFunctions}).

To see that this is fundamentally linked to counting process martingales,
let $X$ have distribution function $F$ on ${\mathbb R}_+$, and define a one-jump
counting process $\{ {\mathbb N}(t) : \ t \ge 0 \}$ by
\begin{eqnarray*}
{\mathbb N}(t) = 1_{[X \le t]} .
\end{eqnarray*}
This process is (trivially) seen to be nondecreasing in $t$ with probability $1$,
and hence is a sub-martingale (a process increasing in conditional mean).
By the Doob-Meyer decomposition theorem there is an increasing predictable
process $\{ {\mathbb A} (t) : \ t \ge 0 \}$ such that
\begin{eqnarray*}
 {\mathbb N} (t) = {\mathbb M} (t) + {\mathbb A} (t)
 \end{eqnarray*}
 where  $\{ {\mathbb M} (t) : \  t \ge 0 \} $ is a mean$-0$ martingale.  In fact for this simple
 counting process it is well-known that
 \begin{eqnarray*}
 {\mathbb A} (t) = \int_{[0,t]} 1_{[X \ge s]} d \overline{\Lambda} (s)
 \end{eqnarray*}
 (see e.g. Appendix B of
 \cite{MR3396731},
 or Chapter 18 of
 \cite{MR0488267}),  %
  and hence we see that
 \begin{eqnarray*}
 {\mathbb M} (t) =  {\mathbb N} (t) - \int_{[0,t]} 1_{[X \ge s]} d \overline{\Lambda} (s)  .
 \end{eqnarray*}
 Comparing this with the identity (\ref{L-Operator-Identity}) rewritten for a distribution function $F$
 on ${\mathbb R}_+$ we see that with $\psi_t  (x) = 1_{[x\le t]} $ and evaluating the resulting identity
 at $x = X $ we get
\begin{eqnarray*}
L \psi_t (X) =  1_{[X \le t]} - \int_0^t 1_{[X \ge y]} d \overline{\Lambda} (y) = {\mathbb M} (t)
\end{eqnarray*}
where $\overline{\Lambda} (x)   \equiv   \int_{[0,x]}  (1-F(y-) )^{-1} d F(y) $ is the cumulative
hazard function corresponding to $F$ on ${\mathbb R}_+$.

But there are still more martingales in this setting which can be represented in terms of the
martingale ${\mathbb M}$ by bringing in the residual life operator $R = I - \overline{H}_F$.
Consider the increasing family of
$\sigma-$fields $\{ {\cal F}_t : \ t \ge 0 \}$ given by
${\cal F}_t \equiv \sigma \{ 1_{[X\le s]} : \ 0 \le s \le t \}$.
Now let $\psi \in L_2^0 (F)$ and consider the process
$$
{\mathbb Y} (t) \equiv  E \{ \psi (X) | {\cal F}_t \},  \qquad t \ge 0 .
$$
Since the $\sigma-$fields $\{ {\cal F}_t \}_{t \ge 0}$ are nested,
$\{ {\mathbb Y}(t) : \ t \ge 0 \}$ is a  martingale  (and it is often called ``Doob's martingale'').
Furthermore, it can be represented in terms of the basic martingale ${\mathbb M}$
using the fundamental identity  $L \circ R = I$  on $L_2^0(F) $ discussed above:
since $\psi = L \circ R \psi $ we see that
\begin{eqnarray*}
{\mathbb Y} (t) & = & E \{  \psi | {\cal F}_t \} =  E \{ L\circ R  \psi | {\cal F}_t \}
 =  \int_{[0,t]}  R \psi (s) d {\mathbb M} (s) .
\end{eqnarray*}

This set of connections deserves to be explored further.
In particular we conjecture that many of the interesting properties
of the classical Hardy operator $H$ and the dual Hardy operator $H^*$ established in the
series of papers by
\cite{MR1676324},  
\cite{MR2390520},  
\cite{MR2747011},   
\cite{MR3180926},    
\cite{MR3868629},   
\cite{MR4051866},   
and \cite{MR4088501}   
will have useful analogues for $\overline{H}_F$ and $\overline{H}_F^*$
in the probability setting for Hardy's inequalities which we have considered here.
On the other hand, the martingale connections of the operators  $L$ and $R$ perhaps
deserve to be better known in the world of classical Hardy type inequalities.

For further explanation of the connections of these processes with right and left censored data
problems in survival analysis, see the Appendix,  Section ~\ref{sec:Appendix} .

If $X_1, \ldots , X_n$ are i.i.d. with (continuous distribution function) $F$,
then
$$
{\mathbb N}_n (t) \equiv \sum_{i=1}^n 1_{[X_i \le t]}  = n {\mathbb F}_n (t)
$$
is a counting process which is simply the sum of independent counting processes
and the sum of the corresponding counting process martingales is again a counting process martingale:
$$
{\mathbb M}_n (t) \equiv \sum_{i=1}^n {\mathbb M}_i (t) = {\mathbb N}_n (t) - \int_0^t  \YY_n (s) \overline{\Lambda} (s)
$$
where $\YY_n (t) \equiv \sum_{i=1}^n 1_{[X_i \ge t]} $ is the number of $X_i$'s ``at risk'' at time $t$.

\section{Proofs}
\label{sec:proofs}

\subsection{Proofs for Section \ref{sec:Mr}}

In order to prove our random variable version of Hardy's inequality we need a Lemma.
This Lemma
has the same structure as
Broadbent's proof of Hardy's inequality (\ref{Hardy}), which is a slightly improved version of
Elliot's proof; see \cite{MR1574000}, 
\cite{MR1574962}, and 
\cite{MR0046395}, page 240.

\begin{Lemma}\label{Broadbent}
Let $a_i$ and $p_i$ be nonnegative numbers for $i=1, \dots,m,$ with $p_1 >0.$
For $p>1$ the inequality
\begin{equation}\label{Broadbent1}
\sum_{n=1}^m \left( \frac {\sum_{i=1}^n a_i p_i}{\sum_{i=1}^n p_i} \right)^p p_n
\leq \left( \frac p{p-1} \right)^p \sum_{n=1}^m a_n^p p_n
\end{equation}
holds.
\end{Lemma}
With $p_i=1$ this inequality is a finite sum version of the discrete Hardy inequality
(\ref{DiscHardyIneq}). Taking limits as  $m \to \infty$ first on  the right hand side and subsequently
on the left hand side of (\ref{Broadbent1}) with $p_i=1$ we obtain the discrete Hardy inequality itself.

\begin{proof} {\em of Lemma \ref{Broadbent}} \\
With the notation
$P_n = \sum_{i=1}^n p_i,\, A_n = \sum_{i=1}^n a_i p_i,\, B_n = A_n / P_n,\,n=1,\dots,m,\\ A_0=B_0=P_0=0$
we rewrite
\begin{equation}\label{NB1}
a_n p_n B_n^{p-1} = \left( A_n - A_{n-1} \right) B_n^{p-1} =  \left(P_n B_n -P_{n-1} B_{n-1} \right) B_n^{p-1}
\end{equation}
into
\begin{equation}\label{NB2}
P_n B_n^p = a_n p_n B_n^{p-1} + P_{n-1} B_{n-1} B_n^{p-1}.
\end{equation}
By Young's inequality ($uv \le u^p/p + v^{p'}/p'$ with $1/p + 1/p' =1$), this implies
\begin{equation}\label{NB3}
P_n B_n^p \leq a_n p_n B_n^{p-1} + P_{n-1} \left( \frac 1p B_{n-1}^p + \frac{p-1}p B_n^p \right)
\end{equation}
and hence
\begin{equation}\label{NB4}
\left( P_n - \frac{p-1}p P_{n-1} \right) B_n^p \leq a_n p_n B_n^{p-1} + \frac 1p P_{n-1} B_{n-1}^p.
\end{equation}
Summing this inequality over $n$ we obtain
\begin{equation}\label{NB5}
\sum_{n=1}^{m} P_n B_n^p - \frac{p-1}p  \sum_{n=1}^{m} P_{n-1} B_n^p
\leq \sum_{n=1}^{m} a_n p_n B_n^{p-1} + \frac 1p \sum_{n=1}^{m -1} P_n B_n^p,
\end{equation}
which is equivalent to
\begin{equation}\label{NB6}
\frac 1p P_{m} B_{m}^p + \frac{p-1}p \sum_{n=1}^{m} \left( P_n - P_{n-1} \right) B_n^p \leq \sum_{n=1}^{m} a_n p_n B_n^{p-1}.
\end{equation}
By H\"older's inequality this yields
\begin{equation}\label{NB6}
\frac{p-1}p \sum_{n=1}^{m} p_n B_n^p \leq \left( \sum_{n=1}^{m} a_n^p p_n \right)^{1/p}
\left( \sum_{n=1}^{m} B_n^p p_n \right)^{(p-1)/p}
\end{equation}
and hence
\begin{equation}\label{NB6}
\left( \sum_{n=1}^{m} B_n^p p_n \right)^{1/p} \leq \frac p{p-1} \left( \sum_{n=1}^{m} a_n^p p_n \right)^{1/p}
\end{equation}
and (\ref{Broadbent1}).
\end{proof}

\begin{proof}{\em {of Theorem \ref{HardyIneq}.}}\\
For large $N$ we define $y_{N,i} = F^{-1}(i/N),\ i = 0, \dots,N-1,\ y_{N,N} = \infty,$
and we apply Lemma \ref{Broadbent} with $m=N$ and
\begin{equation}\label{NP1}
p_n = \int_{(y_{N,n-1},y_{N,n}]} dF,\quad a_n = \int_{(y_{N,n-1},y_{N,n}]} \psi dF/p_n,\quad n=1,\dots,N.
\end{equation}
By Jensen's inequality we have
\begin{equation}\label{NP2}
a_n^p \leq \int_{(y_{N,n-1},y_{N,n}]} \psi^p dF/p_n,\quad n=1,\dots,N,
\end{equation}
and hence
\begin{equation}\label{NP3}
\sum_{n=1}^N a_n^p p_n \leq \sum_{n=1}^N \int_{(y_{N,n-1},y_{N,n}]} \psi^p dF =  E\left( \psi^p(Y) \right).
\end{equation}
For any $x \in {\mathbb R}$ there exists an index $n(N,x)$ with $x \in (y_{N,n(N,x)-1}, y_{N,n(N,x)}]$.
Consequently we have
\begin{eqnarray}\label{NP4}
\lefteqn{\sum_{n=1}^N \left( \int_{(-\infty,y_{N,n}]} \psi dF / F(y_{N,n}) \right)^p {\bf 1}_{(y_{N,n-1},y_{N,n}]}(x) } \\
&& =  \left( \int_{(-\infty,y_{N,n(N,x)}]} \psi dF / F(y_{N,n(N,x)}) \right)^p
\geq \left( \int_{(-\infty,x]} \psi dF / F(y_{N,n(N,x)}) \right)^p \nonumber
\end{eqnarray}
and hence by Tonelli's theorem, Fatou's lemma and the right continuity of $F$
\begin{eqnarray}\label{NP5}
\lefteqn{\liminf_{N \to \infty} \sum_{n=1}^N \left( \frac {\sum_{i=1}^n a_i p_i}{\sum_{i=1}^n p_i} \right)^p p_n } \\
&& = \liminf_{N \to \infty} \sum_{n=1}^N \int _{\mathbb R} \left( \int_{(-\infty,y_{N,n}]} \psi dF / F(y_{N,n}) \right)^p
{\bf 1}_{(y_{N,n-1},y_{N,n}]}(x) dF(x) \nonumber \\
&& = \liminf_{N \to \infty} \int _{\mathbb R} \sum_{n=1}^N \left( \int_{(-\infty,y_{N,n}]} \psi dF / F(y_{N,n}) \right)^p
{\bf 1}_{(y_{N,n-1},y_{N,n}]}(x) dF(x) \nonumber \\
&& \geq \int _{\mathbb R} \liminf_{N \to \infty} \sum_{n=1}^N \left( \int_{(-\infty,y_{N,n}]} \psi dF / F(y_{N,n}) \right)^p
{\bf 1}_{(y_{N,n-1},y_{N,n}]}(x) dF(x) \nonumber \\
&& \geq \int _{\mathbb R} \liminf_{N \to \infty} \left( \int_{(-\infty,x]} \psi dF / F(y_{N,n(N,x)}) \right)^p dF(x) \nonumber \\
&& = \int _{\mathbb R} \left( \int_{(-\infty,x]} \psi dF / F(x) \right)^p dF(x)
= E\left( \left[ \frac{E\left( \psi(Y) {\bf 1}_{[Y \leq X]} \mid X \right)}{F(X)} \right]^p \right). \nonumber
\end{eqnarray}
Combining (\ref{NP5}), Lemma \ref{Broadbent} and (\ref{NP3}) we arrive at a proof of Theorem \ref{HardyIneq}.
\end{proof}

\subsection{Proofs for Section \ref{Hiww}}
\begin{proof} {\em of Theorem \ref{LiMao}}. \
If $B$ equals infinity, inequality (\ref{LiMao1}) is trivial.
So, we may assume that $B$ is finite and hence for any $r \in {\mathbb R}$ that $\mu([r,\infty)) = \infty$ implies $\nu((-\infty,r]) = 0$.
Define
\begin{equation}\label{LiMao13}
{\cal R} = \{ r \,:\, \mu([r,\infty)) < \infty,\ r \in {\mathbb R} \},\quad R_0 = \inf {\cal R}
\end{equation}
and choose $R \geq R_0$.
If ${\cal R}=[R_0,\infty)$ holds, then without loss of generality we may assume that $\mu$ is a finite Borel measure and we take $R=R_0$.
However, if ${\cal R}=(R_0,\infty)$ holds, then we have $\mu([R_0,\infty)) = \infty$ and we take $R > R_0$.
Furthermore, define
\begin{equation}\label{LiMao9}
S_0 = \sup \{ s \,:\, \mu([s,\infty)) >0,\ s \in {\mathbb R} \}
\end{equation}
and note that $S_0 = \infty$ might hold.
If $S_0 = -\infty$ holds, $\mu$ is the null measure and inequality (\ref{LiMao1}) is trivial.
Let $S \leq S_0$ be such that $M_S = \mu([S, \infty)) > 0$ holds.

We introduce the finite measure $\mu_{R,S}$ that has no mass on $(-\infty,R) \cup (S, \infty)$,
equals $\mu$ on the interval $[R,S)$ and has mass $M_S$ at the point $S$.
It has total mass $M_{R,S} = \mu([R,\infty))$ and ``scaled" distribution function
\begin{equation}\label{LiMao14}
F_{R,S}(x) = \mu([R, x])/M_{R,S} {\bf 1}_{[x<S]} + {\bf 1}_{[x \geq S]},\quad x \in {\mathbb R},
\end{equation}
with inverse
\begin{equation}\label{LiMao15}
F_{R,S}^{-1}(u) = \inf \{ x \,:\, F_{R,S}(x) \geq u \},\quad u \in [0,1].
\end{equation}
For $0< \varepsilon < 1$ we define $\delta = \varepsilon M_S / (M_{R,S} \vee 1)$.
With $N=\lceil 1/\delta \rceil$ we choose
\begin{equation}\label{LiMao16}
y_n = F_{R,S}^{-1}(n/N),\ 1 \leq n \leq N-1,\quad y_0 = R,\quad y_N = \infty.
\end{equation}
Note that $(y_{n-1},y_n)$ might be empty, i.e. $y_{n-1}=y_n$. \\
In view of $1/N \leq \delta = \varepsilon M_S / (M_{R,S} \vee 1) < M_S = \mu_{R,S}(\{ S \})$ we have $y_{N-1}=S$
and hence $\mu_{R,S}((y_{N-1}, \infty)) = \mu_{R,S}((S,\infty)) = 0$.

By Theorem 1.4 of \cite{MR3405817} we have for nonnegative $a_i, u_i, v_i,\ i=1, \dots,N$,
\begin{equation}\label{LiMao4}
\left\{ \sum_{n=1}^N \left( \sum_{i=1}^n a_i \right)^q u_n \right\}^{1/q}
\leq k_{q,p}\, B_{\rm d} \, \left\{ \sum_{i=1}^N a_i^p v_i \right\}^{1/p}
\end{equation}
with
\begin{equation}\label{LiMao5}
B_{\rm d} = \max_{1 \leq n \leq N} \left( \sum_{j=n}^N u_j \right)^{1/q} \left( \sum_{i=1}^n v_i^{-1/(p-1)} \right)^{(p-1)/p} .
\end{equation}
With $R = y_0 \leq y_1 \leq \cdots \leq y_N = \infty $ as in (\ref{LiMao16}) we choose
$a_i = \int_{(y_{i-1},y_i]} \psi d\nu,\\ u_i = \int_{(y_{i-1},y_i]} d\mu_{R,S},\
v_i = \left( \int_{(y_{i-1},y_i]} d\nu \right)^{1-p}$ with $v_i = 0$ if $\int_{(y_{i-1},y_i]} d\nu = 0, \ i=2, \dots, N$,
and $a_1 = \int_{[R,y_1]} \psi d\nu ,\ u_1 = \int_{[R,y_1]} d\mu_{R,S},\ v_1 = \left( \int_{[R,y_1]} d\nu \right)^{1-p}$
with $v_1 = 0$ if $\int_{([R,y_1]} d\nu = 0$. \\
With these choices the left hand side of (\ref{LiMao4}) to the power $q$ satisfies
\begin{eqnarray}\label{LiMao6}
\lefteqn{ \sum_{n=1}^N \left( \sum_{i=1}^n a_i \right)^q u_n
= \sum_{n=2}^N \int_{(y_{n-1},y_n]} \left( \int_{[R,y_n]} \psi d\nu \right)^q d\mu_{R,S} \nonumber } \\
&& \hspace{5em} + \int_{[R,y_1]} \left( \int_{[R,y_1]} \psi d\nu \right)^q d\mu_{R,S} \\
&& \geq \sum_{n=2}^N \int_{(y_{n-1},y_n]} \left( \int_{[R,x]} \psi d\nu \right)^q d\mu_{R,S}(x)
          + \int_{[R,y_1]} \left( \int_{[R,x]} \psi d\nu \right)^q d\mu_{R,S}(x) \nonumber \\
&&         =  \int_{[R,\infty)} \left( \int_{[R,x]} \psi d\nu \right)^q d\mu_{R,S}(x)
         =  \int_{\mathbb R} \left( \int_{[R,x]} \psi d\nu \right)^q d\mu_{R,S}(x) \nonumber \\
&& =  \int_{\mathbb R} \left( \int_{[R,x \wedge S]} \psi d\nu \right)^q d\mu_{R,S}(x)
   =  \int_{\mathbb R} \left( \int_{[R,x \wedge S]} \psi d\nu \right)^q d\mu(x). \nonumber
\end{eqnarray}
Furthermore, by Jensen's inequality (or H\"older) the third factor at the right hand side of (\ref{LiMao4}) to the power $p$ satisfies
\begin{eqnarray}\label{LiMao7}
\lefteqn{  \sum_{i=1}^N a_i^p v_i =  \sum_{i=2}^N \left( \int_{(y_{i-1},y_i]} \psi d\nu \right)^p v_i
          + \left( \int_{[R,y_1]} \psi d\nu \right)^p v_1 } \\
&& \leq \sum_{i=2}^N \int_{(y_{i-1},y_i]} \psi^p d\nu \left( \int_{(y_{i-1},y_i]} d\nu \right)^{p-1} v_i
     + \int_{[R,y_1]} \psi^p d\nu \left( \int_{[R,y_1]} d\nu \right)^{p-1} v_1 \nonumber \\
&& = \int_{[R,\infty)} \psi^p d\nu \leq \int_{\mathbb R} \psi^p d\nu, \nonumber
\end{eqnarray}
where the last expression equals the third factor at the right hand side of (\ref{LiMao1}) to the power $p$.
With these choices $B_{\rm d}$ from (\ref{LiMao5}) becomes
\begin{eqnarray}\label{LiMao8}
\lefteqn{ B_{\rm d} = B(y_0,\dots , y_N) \nonumber } \\
&& = \max \left\{ \left( \mu_{R,S}([R,\infty) \right)^{1/q} \left( \nu(-\infty, y_1]) \right)^{(p-1)/p}, \right. \\
&& \hspace{5em} \left. \max_{2 \leq n \leq N} \left( \mu_{R,S}((y_{n-1},\infty) \right)^{1/q} \left( \nu(-\infty, y_n]) \right)^{(p-1)/p} \right\}. \nonumber
\end{eqnarray}
For $2 \leq n \leq N-1$ we have
\begin{eqnarray}\label{LiMao10}
\lefteqn{ \mu_{R,S}((y_{n-1},\infty)) =  \mu_{R,S}([y_n, \infty)) + \mu_{R,S}((y_{n-1},y_n)) } \\
&& \leq \mu_{R,S}([y_n, \infty))\left[ 1 + \frac {\mu_{R,S}((y_{n-1},y_n))}{M_S} \right]
\leq \mu_{R,S}([y_n, \infty))\left[ 1 + \frac {M_{R,S}}{N M_S} \right] \nonumber \\
&&  \leq \mu_{R,S}([y_n, \infty)) [ 1 + \varepsilon ] \nonumber
\end{eqnarray}
and analogously we obtain
\begin{equation}\label{LiMao25}
\mu_{R,S}([R,\infty)) \leq \mu_{R,S}([y_1, \infty)) [ 1 + \varepsilon ].
\end{equation}
This implies that $B_{\rm d}$ from (\ref{LiMao8}) becomes [recall $\mu_{R,S}((y_{N-1}, \infty)) = 0$]
\begin{eqnarray}\label{LiMao11}
\lefteqn{ B_{\rm d} = B(y_0,\dots , y_N) \nonumber } \\
&& \leq [1+\varepsilon] \max_{1 \leq n \leq N-1} \left( \mu_{R,S}([y_n,\infty) \right)^{1/q} \left( \nu(-\infty, y_n]) \right)^{(p-1)/p} \\
&& \leq [1+\varepsilon] \sup_{R \leq x \leq S} \left( \mu([x,\infty) \right)^{1/q} \left( \nu(-\infty, x]) \right)^{(p-1)/p}
   \leq [1+\varepsilon] B, \nonumber
\end{eqnarray}
where $B$ is as in (\ref{LiMao2}).
Since $\varepsilon$ may be chosen arbitrarily close to 0, this implies together with (\ref{LiMao14}) through (\ref{LiMao7}) that inequality (\ref{LiMao1}) holds with the left hand side replaced by the right hand side of (\ref{LiMao6}) to the power $1/q$.

In the case of $R>R_0$ we have $\mu([R_0,\infty)) = \infty$ and hence $\nu((-\infty,R_0]) = 0$ and monotone convergence shows that the right hand side of (\ref{LiMao6}) satisfies
\begin{eqnarray}\label{LiMao21}
\lefteqn{ \lim_{R \downarrow R_0} \int_{\mathbb R} \left( \int_{[R,x \wedge S]} \psi d\nu \right)^q d\mu(x)
           =  \int_{\mathbb R} \left( \int_{(R_0,x \wedge S]} \psi d\nu \right)^q d\mu(x) \nonumber } \\
&& \hspace{10em} = \int_{\mathbb R} \left( \int_{(-\infty,x \wedge S]} \psi d\nu \right)^q d\mu(x).
\end{eqnarray}
In the case of $R=R_0$ we have $\nu((-\infty,R_0)) = 0$ and hence the right hand side of (\ref{LiMao6}) equals
\begin{equation}\label{LiMao22}
\int_{\mathbb R} \left( \int_{[R_0,x \wedge S]} \psi d\nu \right)^q d\mu(x)
= \int_{\mathbb R} \left( \int_{(-\infty,x \wedge S]} \psi d\nu \right)^q d\mu(x).
\end{equation}
In the case of $\mu([S_0,\infty))=\mu(\{S_0\}) >0$ we may choose $S=S_0$ and the right hand side of (\ref{LiMao22}) equals
\begin{equation}\label{LiMao23}
\int_{\mathbb R} \left( \int_{(-\infty,x \wedge S_0]} \psi d\nu \right)^q d\mu(x)
= \int_{\mathbb R} \left( \int_{(-\infty,x]} \psi d\nu \right)^q d\mu(x).
\end{equation}
In the case of $S_0 = \infty$ or $S_0 < \infty,\ \mu([S_0,\infty)) =0$ we choose $S < S_0$ and monotone convergence shows that the right hand side of (\ref{LiMao22}) satisfies
\begin{eqnarray}\label{LiMao24}
\lefteqn{ \lim_{S \uparrow S_0} \int_{\mathbb R} \left( \int_{(-\infty,x \wedge S]} \psi d\nu \right)^q d\mu(x)
  = \lim_{S \uparrow S_0} \int_{\mathbb R} \left( \int_{(-\infty,x] \cap (-\infty,S]} \psi d\nu \right)^q d\mu(x) \nonumber } \\
&&   =  \int_{\mathbb R} \left( \int_{(-\infty,x] \cap (-\infty,S_0)} \psi d\nu \right)^q d\mu(x)
      = \int_{\mathbb R} \left( \int_{(-\infty,x]} \psi d\nu \right)^q d\mu(x).
\end{eqnarray}
Since inequality (\ref{LiMao1}) holds with the left hand side replaced by the right hand side of (\ref{LiMao6}) to the power $1/q$,
the above argument involving (\ref{LiMao21}) through (\ref{LiMao24}) completes the proof of (\ref{LiMao1}) and the theorem.
\end{proof}

For the proof of Theorem \ref{MuckenhouptII} we need the following Lemma.

\begin{Lemma}\label{LemmaMII}
For $F$ and $G$ distribution functions, $\chi$ a nonnegative measurable function and $0< \gamma <1$ we have
\begin{equation}\label{M3II}
\gamma \int_{(-\infty,x]} \chi(y) \left[ \int_{(-\infty,y]} \chi dG \right]^{\gamma -1} dG(y)
\leq  \left[ \int_{(-\infty,x]} \chi dG \right]^\gamma
\end{equation}
and
\begin{equation}\label{M4II}
\gamma \int_{[y,\infty)} \chi(x) \left[ \int_{[x,\infty)} \chi dF \right]^{\gamma -1} dF(x)
\leq  \left[ \int_{[y,\infty)} \chi dF \right]^\gamma.
\end{equation}
\end{Lemma}
\begin{proof}
By symmetry it suffices to prove (\ref{M3II}), which with the distribution function
$G_x(y) =  \int_{(-\infty,y\wedge x]} \chi dG /  \int_{(-\infty,x]} \chi dG$ is equivalent to
\begin{equation}\label{M5II}
\gamma \int_{-\infty}^\infty G_x^{\gamma -1} dG_x \leq 1.
\end{equation}
With the random variable $U$ uniformly distributed on the unit interval the left hand side of this inequality equals and satisfies
\begin{equation}\label{M6II}
\gamma E\left( \left[G_x \left( G_x^{-1}(U) \right) \right]^{\gamma -1} \right) \leq \gamma E \left( U^{\gamma -1} \right) = 1.
\end{equation}
\end{proof}
\begin{proof}(of Theorem \ref{MuckenhouptII}).
The choice $\psi(y) = V^{-1/(p-1)}(y) {\bf 1}_{[y \leq r]}$ in inequality (\ref{M1II}) leads to the string of (in)equalities
\begin{eqnarray}\label{M11II}
\lefteqn{ \left[ \int_{(-\infty, x]}V^{-1/(p-1)} dG \right]^p \int_{[x,\infty)} U dF \nonumber } \\
&& = E\left(  \left[ E\left( V^{-1/(p-1)}(Y) {\bf 1}_{[Y \leq x]} \right) \right]^{p} U(X) {\bf 1}_{[X \geq x]} \right) \\
&& \leq E\left(  \left[ E\left( V^{-1/(p-1)}(Y) {\bf 1}_{[Y \leq x]} {\bf 1}_{[Y \leq X]} \mid X \right) \right]^{p} U(X) \right) \nonumber \\
&& \leq C\, E \left( V^{-1/(p-1)}(Y) {\bf 1}_{[Y \leq x]} \right) = C \,  \int_{(-\infty, x ]}V^{-1/(p-1)} dG ,\quad  x \in {\mathbb R}, \nonumber
\end{eqnarray}
which implies the first inequality in (\ref{M3.5II}).  With
\begin{equation}\label{M7II}
h(y) = V^{1/p}(y) \left[\int_{(-\infty,y]} V^{-1/(p-1)} dG \right]^{(p-1)/p^2}
\end{equation}
inequality (\ref{M3II}) of Lemma \ref{LemmaMII} with $\chi = V^{-1/(p-1)}$ and $\gamma = 1-1/p = (p-1)/p$ yields
\begin{eqnarray}\label{M8II}
\lefteqn{ E \left( h^{-p/(p-1)}(Y) {\bf 1}_{[Y \leq x]} \right) \nonumber} \\
&& = \int_{(-\infty,x]} V^{-1/(p-1)}(y) \left[ \int_{(-\infty,y]} V^{-1/(p-1)} dG \right]^{-1/p} dG(y) \nonumber \\
&& \leq \frac p{p-1} \left[ \int_{(-\infty,x]} V^{-1/(p-1)} dG \right]^{(p-1)/p}.
\end{eqnarray}
By H\"older's inequality this implies
\begin{eqnarray}\label{M9II}
\lefteqn{ E\left(  \left[ E\left( \psi(Y) {\bf 1}_{[Y \leq X]} \mid X \right) \right]^{p} U(X) \right) \nonumber } \\
&& = E\left(  \left[  E\left( \psi(Y) h(Y) (h(Y))^{-1} {\bf 1}_{[Y \leq X]} \mid X \right) \right]^{p} U(X) \right) \nonumber \\
&& \leq E \Bigg( E \left(\psi^p(Y)h^p(Y)  {\bf 1}_{[Y \leq X]} \mid X \right) \nonumber \\
&& \qquad \left. \left[ E \left( h^{-p/(p-1)}(Y) {\bf 1}_{[Y \leq X]} \mid X \right) \right]^{p-1} U(X) \right) \\
&& \leq \left(\frac p{p-1} \right)^{p-1} E \Bigg( E \left( \psi^p(Y)h^p(Y)  {\bf 1}_{[Y \leq X]} \mid X \right) \nonumber \\
&& \qquad \left. \left[ \int_{(-\infty,X]} V^{-1/(p-1)} dG \right]^{(p-1)^2/p} U(X) \right) \nonumber \\
&& = \left(\frac p{p-1} \right)^{p-1} E \left( \psi^p(Y)h^p(Y)
 E \left( \left[ \int_{(-\infty,X]} V^{-1/(p-1)} dG \right]^{(p-1)^2/p} \right. \right. \nonumber \\
&& \hspace{10em} U(X){\bf 1}_{[Y \leq X]} \mid Y \Bigg) \Bigg). \nonumber
\end{eqnarray}
By the definition of $B$ in (\ref{M2II}) the right hand side of (\ref{M9II}) is bounded from above by
\begin{eqnarray}\label{M10II}
\lefteqn{ \left(\frac p{p-1} \right)^{p-1} B^{(p-1)/p} E \left( \psi^p(Y) V(Y)  \left[\int_{(-\infty,Y]} V^{-1/(p-1)} dG \right]^{(p-1)/p}
\right. \nonumber } \\
&& E \left( \left[ \int_{[X,\infty)} U dF \right]^{(1/p)-1} U(X){\bf 1}_{[Y \leq X]} \mid Y \right) \Bigg) \\
&& \leq \frac {p^p}{(p-1)^{p-1}} B^{(p-1)/p} E \left( \psi^p(Y) V(Y)  \left[\int_{(-\infty,Y]} V^{-1/(p-1)} dG \right]^{(p-1)/p} \right.
 \nonumber \\
&& \hspace{10em} \left[ \int_{[Y,\infty)} U dF \right]^{1/p} \Bigg), \nonumber
\end{eqnarray}
where the inequality follows from (\ref{M4II}) of Lemma \ref{LemmaMII}.
By the definition of $B$ the last expression is bounded by the right hand side of (\ref{M3.5II}), which completes the proof of (\ref{M3.5II}).
\end{proof}

\subsection{Proofs for Section \ref{ArHi}}

\begin{proof} {\em of Theorem~\ref{ReverseHardyIneq}.}
Let $f$ be a density of $F$.
The monotonicity of $\psi$ implies
\begin{eqnarray}\label{5a}
\frac d{dx} \left[ \int_{-\infty}^x \psi(y) dF(y)\right]^p
& = & p  \left[ \int_{-\infty}^x \psi(y) dF(y)\right]^{p-1} \psi(x)f(x) \nonumber  \\
& \geq & p \psi^p(x) F^{p-1}(x) f(x)
\end{eqnarray}
for Lebesgue almost all $x \in {\mathbb R}$.
So we have
\begin{equation}\label{5b}
\left[ \int_{-\infty}^x \psi(y) dF(y)\right]^p \geq p \int_{-\infty}^x \psi^p(y) (F(y))^{p-1} dF(y)
\end{equation}
and hence
\begin{eqnarray}\label{5c}
\lefteqn{ E\left( \left[ \frac{E\left( \psi(Y) {\bf 1}_{[Y \leq X]} \mid X \right)}{F(X)} \right]^p \right) \nonumber } \\
&& \geq  p \int_{-\infty}^\infty \int_{-\infty}^x \psi^p(y) F^{p-1}(y) dF(y) F^{-p}(x) dF(x) \nonumber \\
&& = p \int_{-\infty}^\infty \int_y^\infty F^{-p}(x)f(x) dx \psi^p(y) (F(y))^{p-1} dF(y) \\
&& = \frac p{p-1}\int_{-\infty}^\infty\left[F^{1-p}(y) -1 \right] \psi^p(y) (F(y))^{p-1} dF(y) \nonumber \\
&& = \frac p{p-1} E\left(\psi^p(Y) \left(1- F^{p-1}(Y) \right) \right), \nonumber
\end{eqnarray}
which is the first inequality of (\ref{ReverseDensityHardy}).
Since $\psi^p$ and $1-F^{p-1}$ are both nonincreasing, $\psi^p(Y)$ and $1- F^{p-1}(Y)$
are nonnegatively correlated and consequently their covariance is nonnegative implying
\begin{eqnarray}\label{5d}
E\left(\psi^p(Y) \left(1- F^{p-1}(Y) \right) \right)
& \geq & E\left(\psi^p(Y)\right) E\left(1-  F^{p-1}(Y) \right) \nonumber  \\
& = & \frac{p-1}p  E\left(\psi^p(Y)\right).
\end{eqnarray}
This results in the second inequality of (\ref{ReverseDensityHardy}).

Note that inequality (\ref{ReverseGeneralHardy}) and hence the inequality between
the left hand side and the right hand side of (\ref{ReverseDensityHardy}) is obvious as $\psi$ is nonincreasing.

Let $F$ be general and $p$ integer.
As $X_1, \dots, X_p$ are independent and identically distributed and
$\psi(\cdot){\bf 1}_{[\cdot \leq x]}$ is nonincreasing, we have
\begin{equation}\label{5e}
E\left(\prod_{i=1}^p \psi(X_i){\bf 1}_{[X_i \leq x]} \right) \geq E \left(\psi^p(X_{(p)}) {\bf1}_{[X_{(p)} \leq x]} \right)
\end{equation}
and hence
\begin{equation}\label{5f}
E\left( \left[ \frac{E\left( \psi(Y) {\bf 1}_{[Y \leq X]} \mid X \right)}{F(X)} \right]^p \right)
\geq E \left(\psi^p(X_{(p)}) {\bf1}_{[X_{(p)} \leq X]} F^{-p}(X) \right),
\end{equation}
which implies (\ref{ReverseIntegerHardy}).
\end{proof}

\begin{proof} of Corollary \ref{ReverseInequalities}.
Let $X$ and $Y$ be uniformly distributed on the interval $(0,K)$.
Our reverse Hardy inequality (\ref{ReverseDensityHardy}) becomes
\begin{equation}\label{5g}
\frac 1K \int_0^K \left[ \frac 1x \int_0^x \psi(y) dy \right]^p dx
\geq \frac p{p-1}\, \frac 1K \int_0^K \psi^p(y) \left(1-\left(\frac{y}K \right)^{p-1} \right) dy,
\end{equation}
which for $0< \varepsilon \leq 1$ implies
\begin{equation}\label{5h}
\int_0^K \left[ \frac 1x \int_0^x \psi(y) dy \right]^p dx
\geq \frac p{p-1} \int_0^{\varepsilon K} \psi^p(y) \left(1-\varepsilon^{p-1} \right) dy,
\end{equation}
Taking limits for $K \to \infty$ and subsequently $\varepsilon \downarrow 0$ we arrive at (\ref{ReverseContHardyIneq}).

For the second part of the corollary we take $X$ and $Y$ uniformly distributed on $\{1, \dots, K\}$.
In view of $P(X_{(p)} \leq n) = (n/K)^p$ our inequality (\ref{ReverseIntegerHardy}) with $\psi(k)=c_k$ becomes
\begin{equation}\label{5i}
\frac 1K \sum_{n=1}^K \left[ \frac{\frac 1K \sum_{k=1}^n c_k}{n/K} \right]^p
\geq \sum_{n=1}^K c_n^p \left[ \left(\frac nK \right)^p - \left( \frac {n-1}K \right)^p \right] \frac 1K \sum_{k=n}^K \left(\frac kK \right)^{-p},
\end{equation}
which implies
\begin{equation}\label{5j}
\sum_{n=1}^K \left[ \frac 1n \sum_{k=1}^n c_k \right]^p
\geq \sum_{n=1}^K c_n^p \left[ n^p - (n-1)^p \right] \sum_{k=n}^{K_0} \frac 1{k^p}
\end{equation}
for any integer $K_0 \leq K$ and the corresponding sum vanishing for $n>K_0$.
Taking limits as $K \to \infty$ and subsequently $K_0 \to \infty$ we obtain
\begin{equation}\label{5k}
\sum_{n=1}^\infty \left[ \frac 1n \sum_{k=1}^n c_k \right]^p
\geq \sum_{n=1}^\infty c_n^p \left[ n^p - (n-1)^p \right] \sum_{k=n}^\infty \frac 1{k^p}.
\end{equation}
Lemma 2 of \cite{MR854569} 
shows
\begin{equation}\label{5l}
\left[ n^p - (n-1)^p \right] \sum_{k=n}^\infty \frac 1{k^p} \geq \zeta(p)
\end{equation}
for $n \geq 2$.
As for $n=1$ equality holds in (\ref{5l}), the proof that for integer $p$ inequality
(\ref{ReverseDiscHardyIneq}) can be obtained from our inequality (\ref{ReverseIntegerHardy}), is complete.
\end{proof}

\subsection{Proofs for Section \ref{sec:Ci}}

We will use the following Lemma, which shows the structure of Copson's proof
of his Theorem B with sums over infinitely many terms replaced by finite sums; see \cite{MR1574056}.

\begin{Lemma}\label{C1}
Let $a_i$ and $p_i$ be nonnegative numbers for $i=1, \dots,m,$ with $p_1 >0.$
For $p>1$ the inequality
\begin{equation}\label{C2}
\sum_{n=1}^m \left( \sum_{i=n}^m \frac{a_i p_i}{\sum_{j=1}^i p_j} \right)^p p_n \leq p^p \sum_{n=1}^m a_n^p p_n
\end{equation}
holds.
\end{Lemma}
Note that part of Theorem B of \cite{MR1574056} follows from this inequality
by taking limits for $m \to \infty$, first at the right hand side, subsequently within the
$p$-th power at the left hand side, and finally for the first sum at the left hand side.

\begin{proof} {\em of Lemma \ref{C1}.} \\
With the notation
\begin{equation}\label{C3}
P_n = \sum_{i=1}^n p_i,\quad A_n = \sum_{i=n}^m \frac{a_i p_i}{P_i},\quad n=1,\dots,m,\quad P_0 = A_{m+1}=0,
\end{equation}
Young's inequality (as in the proof of Lemma~\ref{Broadbent}) yields
\begin{eqnarray}\label{C4}
\lefteqn{ A_n^p p_n - p A_n^{p-1} a_n p_n = A_n^p p_n - p A_n^{p-1} P_n \left( A_n - A_{n+1} \right) } \\
&& \leq \left( p_n -p P_n \right) A_n^p + P_n \left( (p-1)A_n^p + A_{n+1}^p \right) = P_n A_{n+1}^p - P_{n-1} A_n^p \nonumber
\end{eqnarray}
for $n = 1, \dots, m.$
Summing this inequality over $n$ we obtain
\begin{equation}\label{C5}
\sum_{n=1}^{m} A_n^p p_n - p \sum_{n=1}^{m} a_n A_n^{p-1} p_n \leq 0.
\end{equation}
By H\"older's inequality the second sum in (\ref{C5}) is bounded as follows
\begin{equation}\label{C6}
\left( \sum_{n=1}^{m} a_n A_n^{p-1} p_n \right)^p \leq \sum_{n=1}^{m} a_n^p p_n \left( \sum_{n=1}^{m} A_n^p p_n \right)^{p-1}.
\end{equation}
Together with (\ref{C5}) this implies
\begin{equation}\label{C7}
\left( \sum_{n=1}^{m} A_n^p p_n \right)^p \leq p^p \left( \sum_{n=1}^{m} a_n A_n^{p-1} p_n \right)^p
\leq p^p \sum_{n=1}^{m} a_n^p p_n \left( \sum_{n=1}^{m} A_n^p p_n \right)^{p-1}
\end{equation}
and hence (\ref{C2}).
\end{proof}

\begin{proof} {\em of Theorem \ref{NCopsonIneq}}.  \\
As in the proof of Theorem \ref{HardyIneq} we define $y_{N,i} = F^{-1}(i/N),\ i = 0, \dots,N-1,\ y_{N,N} = \infty,$ for large $N$  and we apply Lemma \ref{C1} with $m=N$, but this time we choose
\begin{equation}\label{NP6}
p_n = \int_{[y_{N,n-1},y_{N,n})} dF,\quad a_n = \int_{[y_{N,n-1},y_{N,n})} \psi dF/p_n,\quad n=1,\dots,N.
\end{equation}
By Jensen's inequality we have
\begin{equation}\label{NP7}
a_n^p \leq \int_{[y_{N,n-1},y_{N,n})} \psi^p dF/p_n,\quad n=1,\dots,N,
\end{equation}
and hence
\begin{equation}\label{NP8}
\sum_{n=1}^N a_n^p p_n \leq \sum_{n=1}^N \int_{[y_{N,n-1},y_{N,n})} \psi^p dF =  E\left( \psi^p(Y) \right).
\end{equation}
Observe that $F(y_{N,i}-) \leq F(y) + 1/N$ holds for $y \in [y_{N,i-1},y_{N,i})$.
Consequently we have
\begin{eqnarray}\label{NP9}
\lefteqn{ \sum_{i=n}^N \frac{a_i p_i}{\sum_{j=1}^i p_j} = \sum_{i=n}^N \int_{[y_{N,i-1},y_{N,i})} \frac{\psi(y)}{F(y_{N,i}-)} dF(y) } \\
&& \geq  \sum_{i=n}^N \int_{[y_{N,i-1},y_{N,i})} \frac{\psi(y)}{F(y) + 1/N} dF(y) = \int_{[y_{N,n-1},\infty)} \frac{\psi(y)}{F(y) + 1/N} dF(y) \nonumber
\end{eqnarray}
and hence by Fatou's lemma
\begin{eqnarray}\label{NP10}
\lefteqn{\liminf_{N \to \infty} \sum_{n=1}^N \left( \sum_{i=n}^N \frac{a_i p_i}{\sum_{j=1}^i p_j} \right)^p p_n \nonumber} \\
&& \geq \liminf_{N \to \infty} \sum_{n=1}^N \int_{[y_{N,n-1},y_{N,n})} \left( \int_{[y_{N,n-1},\infty)} \frac{\psi(y)}{F(y) + 1/N} dF(y) \right)^p dF(x) \nonumber \\
&& \geq \liminf_{N \to \infty} \sum_{n=1}^N \int_{[y_{N,n-1},y_{N,n})} \left( \int_{[x,\infty)} \frac{\psi(y)}{F(y) + 1/N} dF(y) \right)^p dF(x)\\
&& = \liminf_{N \to \infty} \int_{\mathbb R} \left( \int_{[x,\infty)} \frac{\psi(y)}{F(y) + 1/N} dF(y) \right)^p dF(x) \nonumber \\
&& = \int_{\mathbb R} \left( \int_{[x,\infty)} \liminf_{N \to \infty} \frac{\psi(y)}{F(y) + 1/N} dF(y) \right)^p dF(x) \nonumber \\
&& = E\left( \left[ E \left( \frac{\psi(Y)}{F(Y)} {\bf 1}_{[Y \geq X]} \mid X \right) \right]^p \right). \nonumber
\end{eqnarray}
Combining (\ref{NP10}), Lemma \ref{C1} and (\ref{NP8}) we arrive at a proof of Theorem \ref{NCopsonIneq}.
\end{proof}
\begin{proof}{\em of Theorem \ref{equivalence}.} \\
Let $\eta$ be a nonnegative measurable function.
H\"older's inequality and subsequently Hardy's inequality (\ref{Hardy}) yield
\begin{eqnarray}\label{e1}
\lefteqn{  E\left( E \left( \frac{\psi(Y)}{F(Y)} {\bf 1}_{[Y \geq X]} \mid X \right) \eta(X) \right)
=  E\left( \psi(Y) \frac{E \left( \eta(X) {\bf 1}_{[X \leq Y]} \mid Y \right)}{F(Y)} \right) \nonumber } \\
&& \leq \left[E\left(\psi^p(Y) \right) \right]^{1/p}
\left[ E \left( \left[\frac{E \left( \eta(X) {\bf 1}_{[X \leq Y]} \mid Y \right)}{F(Y)} \right]^{p/(p-1)} \right) \right]^{(p-1)/p} \\
&& \leq \left[E\left(\psi^p(Y) \right) \right]^{1/p}
\left[\left( \frac{p/(p-1)}{p/(p-1) -1} \right)^{p/(p-1)} E \left(\eta^{p/(p-1)}(X) \right) \right]^{(p-1)/p} \nonumber \\
&& = p \left[E\left(\psi^p(Y) \right) \right]^{1/p} \left[ E \left(\eta^{p/(p-1)}(X) \right) \right]^{(p-1)/p}. \nonumber
\end{eqnarray}
Taking
\begin{equation}\label{e2}
\eta(X) = \left[ E \left( \frac{\psi(Y)}{F(Y)} {\bf 1}_{[Y \geq X]} \mid X \right) \right]^{p-1}
\end{equation}
we obtain Copson's inequality (\ref{NCopson}) from (\ref{e1}).
Similarly, H\"older's inequality and subsequently Copson's inequality (\ref{NCopson}) yield
\begin{eqnarray}\label{e3}
\lefteqn{  E\left( \frac{ E \left( \psi(Y) {\bf 1}_{[Y \leq X]} \mid X \right)}{F(X)} \eta(X) \right)
=  E\left( \psi(Y) E \left( \frac{\eta(X)}{F(X)} {\bf 1}_{[X \geq Y]} \mid Y \right) \right) \nonumber } \\
&& \leq \left[E\left(\psi^p(Y) \right) \right]^{1/p}
\left[ E \left( \left[ E \left( \frac{\eta(X)}{F(X)} {\bf 1}_{[X \geq Y]} \mid Y \right) \right]^{p/(p-1)} \right) \right]^{(p-1)/p} \\
&& \leq \left[E\left(\psi^p(Y) \right) \right]^{1/p}
\left[\left( \frac p{p-1} \right)^{p/(p-1)} E \left(\eta^{p/(p-1)}(X) \right) \right]^{(p-1)/p} \nonumber \\
&& = \frac p{p-1} \left[E\left(\psi^p(Y) \right) \right]^{1/p} \left[ E \left(\eta^{p/(p-1)}(X) \right) \right]^{(p-1)/p}. \nonumber
\end{eqnarray}
Taking
\begin{equation}\label{e4}
\eta(X) = \left[\frac{ E \left( \psi(Y) {\bf 1}_{[Y \leq X]} \mid X \right)}{F(X)} \right]^{p-1}
\end{equation}
we obtain Hardy's inequality (\ref{Hardy}) from (\ref{e3}).
\end{proof}

\subsection{Proof for Section \ref{sec:ReverseCopsonIneq}}

\begin{proof} of Theorem~\ref{NewReverseCopsonIneq}.
First we prove that for $p \in [1,\infty)$, for arbitrary $F$ and for $x \mapsto \psi(x)/F(x)$ nonincreasing
\begin{equation}\label{NRC7}
E\left( \left[ E \left( \frac{\psi(Y)}{F(Y)} {\bf 1}_{[Y \geq X]} \mid X \right) \right]^p \right)
\geq E\left( \psi^p(Y) \left[\frac {F(Y-)}{F(Y)} \right]^p \right)
\end{equation}
holds.
Observe that for continuous $F$ this implies (\ref{NRC0}).
To prove (\ref{NRC7}) we follow the line of argument in the proof of Theorem 4 of \cite{MR854569}.
For $x < y$ the monotonicity of $\psi/F$ implies
\begin{equation}\label{NRC8}
\int_{[x,y)} \frac \psi F dF \geq \frac {\psi(y)}{F(y)} [F(y-) - F(x-)]
\end{equation}
and hence
\begin{equation}\label{NRC9}
p \left[ \int_{[x,y)} \frac \psi F \, dF \right]^{p-1} \frac {\psi(y)}{F(y)}
\geq p \left[ \frac {\psi(y)}{F(y)} \right]^p [F(y-) - F(x-)]^{p-1}
\end{equation}
and
\begin{eqnarray}\label{NRC10}
\lefteqn{ \int_{\mathbb R} \int_{[x,\infty)} p \left[ \int_{[x,y)} \frac \psi F dF \right]^{p-1} \frac {\psi(y)}{F(y)}\, dF(y)\,dF(x) } \\
&& \geq \int_{\mathbb R} \int_{[x,\infty)} p \left[ \frac {\psi(y)}{F(y)} \right]^p [F(y-) - F(x-)]^{p-1} \, dF(y)\,dF(x). \nonumber
\end{eqnarray}
In view of $F(F^{-1}(u)-) \leq u$ and since $u \leq F(y-)$ implies $F^{-1}(u) \leq y$,
Fubini's theorem shows that the right hand side of (\ref{NRC10}) equals and satisfies
\begin{eqnarray}\label{NRC11}
\lefteqn{ \int_{\mathbb R} \left[ \frac {\psi(y)}{F(y)} \right]^p  \int_{(-\infty,y]} p[F(y-) - F(x-)]^{p-1} \, dF(x)\,dF(y) \nonumber } \\
&& = \int_{\mathbb R} \left[ \frac {\psi(y)}{F(y)} \right]^p \int_0^1 p[F(y-) - F(F^{-1}(u)-)]^{p-1} {\bf 1}_{[F^{-1}(u) \leq y]} \, du\,dF(y) \nonumber \\
&& \geq \int_{\mathbb R} \left[ \frac {\psi(y)}{F(y)} \right]^p \int_0^1 p[F(y-) - u]^{p-1} {\bf 1}_{[u \leq F(y-)]} \, du\,dF(y) \\
&&    = \int_{\mathbb R} \left[ \frac {\psi(y)}{F(y)} \right]^p [F(y-)]^p \, dF(y)
    = E\left( \psi^p(Y) \left[\frac {F(Y-)}{F(Y)} \right]^p \right). \nonumber
\end{eqnarray}
Furthermore, for fixed $x$ we define the distribution function \\
$G_x(y) = \int_{[x,y)} (\psi /F)  \, dF\, / \, \int_{[x,\infty)} (\psi /F) \, dF$ and we obtain
\begin{equation}\label{NRC12}
\int_{-\infty}^\infty p [G_x(y-)]^{p-1} dG_x(y) = \int_0^1 p[G_x(G_x^{-1}(u)-)]^{p-1} du \leq \int_0^1 pu^{p-1} du =1.
\end{equation}
This shows that the left hand side of (\ref{NRC10}) is bounded from above by
\begin{equation}\label{NRC13}
\int_{\mathbb R} \left[\int_{[x,\infty)} \psi /F \, dF \right]^p dF(x) = E\left( \left[ E \left( \frac{\psi(Y)}{F(Y)} {\bf 1}_{[Y \geq X]} \mid X \right) \right]^p \right).
\end{equation}
Combining this with (\ref{NRC10}) and (\ref{NRC11}) we arrive at (\ref{NRC7}) and hence at (\ref{NRC0}).

To prove (\ref{NRC1}) and (\ref{NRC2}) we restrict attention to integer $p$
and let $X,Y,Y_1, \dots, Y_p$ be independent random variables all with distribution function $F$.

If $F$ is continuous,
the monotonicity of $\psi$ implies that
\begin{eqnarray}\label{NRC3}
\lefteqn{ E\left( \left[ E \left( \frac{\psi(Y)}{F(Y)} {\bf 1}_{[Y \geq X]} \mid X \right) \right]^p \right)
    = E\left( \prod_{i=1}^p \left[ E \left( \frac{\psi(Y_i)}{F(Y_i)} {\bf 1}_{[X \leq Y_i]} \mid X \right) \right] \right) \nonumber }\\
&& = E\left( E \left(  \prod_{i=1}^p \frac{\psi(Y_i)}{F(Y_i)} {\bf 1}_{[X \leq Y_i]} \mid X \right) \right)
    = E\left( \prod_{i=1}^p \frac{\psi(Y_i)}{F(Y_i)} {\bf 1}_{[X \leq Y_i]} \right) \nonumber \\
&& = p!\, E\left( \left[\prod_{i=1}^p \frac{\psi(Y_i)}{F(Y_i)} \right] {\bf 1}_{[X \leq Y_1 \leq \cdots \leq Y_p]} \right)
    \geq p!\, E\left( \psi^p(Y_p)\frac{ {\bf 1}_{[X \leq Y_1 \leq \cdots \leq Y_p]} }{F(Y_1) \cdots F(Y_p)} \right) \nonumber \\
&& =  p!\, E\left( \psi^p(Y_p)\frac{ {\bf 1}_{[Y_1 \leq Y_2 \leq \cdots \leq Y_p]} }{F(Y_2) \cdots F(Y_p)} \right) = p!\, E\left( \psi^p(Y) \right),
\end{eqnarray}
where equality holds if $\psi$ is constant.

Similarly, if $F$ is arbitrary, we derive
\begin{eqnarray}\label{NRC4}
\lefteqn{ E\left( \left[ E \left( \frac{\psi(Y)}{F(Y)} {\bf 1}_{[Y \geq X]} \mid X \right) \right]^p \right)
    = E\left( \prod_{i=1}^p \frac{\psi(Y_i)}{F(Y_i)} {\bf 1}_{[X \leq Y_i]} \right) \nonumber } \\
&& \geq E\left( \left[\prod_{i=1}^p \frac{\psi(Y_i)}{F(Y_i)} \right] {\bf 1}_{[X \leq Y_1 \leq \cdots \leq Y_p]} \right)
    \geq E\left( \psi^p(Y_p)\frac{ {\bf 1}_{[X \leq Y_1 \leq \cdots \leq Y_p]} }{F(Y_1) \cdots F(Y_p)} \right) \nonumber \\
&& = E\left( \psi^p(Y_p)\frac{ {\bf 1}_{[Y_1 \leq Y_2 \leq \cdots \leq Y_p]} }{F(Y_2) \cdots F(Y_p)} \right) = E\left( \psi^p(Y) \right).
\end{eqnarray}
One may check that equalities in (\ref{NRC4}) hold if $F$ is degenerate.
\end{proof}

\subsection{Proofs for Section \ref{sec:CarlemanPolyaKnopp}}

\begin{proof} {\em of Theorem~\ref{CarlesonPolya-Knopp}.}
By Hardy's inequality in the probability form (\ref{Hardy}) with $\psi$ replaced by $\psi^{1/p}$  we have
\begin{eqnarray*}
E \left ( \left [ \frac{E ( \psi^{1/p} (Y) 1_{[Y \le X]}  | X )}{ F(X) } \right ]^p \right ) \le \left ( \frac{p}{p-1} \right )^p E (\psi (Y))
\end{eqnarray*}
where $(p/(p-1))^p \rightarrow  e$ as $p\rightarrow \infty$.
Furthermore, taking the logarithm of the expression inside the outer expectation we see
that it is equal to
\begin{eqnarray*}
p \log E \left ( \frac{\psi^{1/p} (Y) 1_{[Y\le X]}}{F(X)}\bigg | X \right )
= \frac{1}{\alpha}
\left\{ \log E \left ( \psi^{\alpha} (Y) 1_{[Y\le X]} \mid X \right) - \log F(X) \right\}
\end{eqnarray*}
after letting $p = 1/\alpha$.
Now for every fixed $X=X(\omega)$ we see that this difference quotient converges as $\alpha \downarrow 0$ by the chain rule as follows:
\begin{eqnarray*}
\lefteqn{ \lim_{\alpha \downarrow 0} \frac{1}{\alpha} \left\{ \log E \left( \psi^{\alpha} (Y) 1_{[Y\le X]} \mid X \right) - \log F(X) \right\}  \nonumber } \\
&& = \frac 1{F(X)} \lim_{\alpha \downarrow 0} E \left( \frac{\psi^{\alpha}(Y) -1}{\alpha} 1_{[Y\le X]} \mid X \right)
    = \frac{E ( 1_{[Y\le X]} \log \psi(Y) \mid X )}{F(X)},
\end{eqnarray*}
where the last equality holds by dominated convergence.
Indeed, for any $1> \varepsilon > \alpha >0$ we have
\begin{eqnarray*}
\left| \frac{\psi^{\alpha}(Y) -1}{\alpha} \right| = \left| \frac 1\alpha \int_0^\alpha \exp(z \log \psi(Y)) dz \log \psi(Y) \right| \leq \left( \psi^\varepsilon(Y) \vee 1 \right) \left|\log \psi(Y) \right|
\end{eqnarray*}
and the right hand side has finite expectation in view of $\psi \in L_1 (F)$ .
\end{proof}

\section{Summary}\label{sec:summary}
Our sharp inequalities related to Hardy's inequality read as follows.
\begin{equation}\label{summaryHardy}
E\left( \psi^p(Y) \right) \leq E\left( \left[ \frac{E\left( \psi(Y) {\bf 1}_{[Y \leq X]} \mid X \right)}{F(X)} \right]^p \right)
\leq \left(\frac p{p-1} \right)^p E\left( \psi^p(Y) \right),
\end{equation}
where the first inequality holds if $F$ is absolutely continuous and $\psi $ is  nonincreasing.

Our sharp inequalities related to Copson's inequality are the following.
\begin{equation}\label{summaryCopson}
\, E\left( \psi^p(Y) \right) \leq E\left( \left[ E \left( \frac{\psi(Y)}{F(Y)} {\bf 1}_{[Y \geq X]} \mid X \right) \right]^p \right)
\leq p^p E\left( \psi^p(Y) \right),
\end{equation}
where the first inequality holds if $F$ is continuous and $x \mapsto \psi(x)/F(x)$ is nonincreasing.\\
Our Hardy inequality with weights and mixed norms is
\begin{eqnarray}\label{summaryLiMao}
\lefteqn{ \left\{E\left(  \left[ E\left( \psi(Y) {\bf 1}_{[Y \leq X]} \mid X \right) \right]^{q} U(X) \right) \right\}^{1/q} \nonumber} \\
&& \leq \left( \frac{(q-p)/p}{{\rm Beta}(p/(q-p), (q-1)p/(q-p))} \right)^{(q-p)/pq} \\
&& \hspace{2em} \sup_{x \in {\mathbb R}} \left[\int_{[x,\infty)} U dF \right]^{1/q} \left[ \int_{(-\infty, x]} V^{-1/(p-1)} dG \right]^{(p-1)/p}
 \left\{E\left( \psi^p(Y) V(Y) \right) \right\}^{1/p}. \nonumber
\end{eqnarray}
Detailed conditions are given in the respective Theorems.

\section{Applications and Related Work}\label{ApplicRelWork}
We close with a few brief comments concerning applications and related work.

As noted by \cite{MR1897417}, 
Hardy's inequality (\ref{DiscHardyIneq}), and especially the weighted version thereof due to
\cite{MR0311856},  
has been applied by
 \cite{MR1710983}  
to obtain useful bounds for the spectral gap for birth-and-death Markov chains.
He provides a nice overview of alternative methods and their potential drawbacks.
 \cite{MR1682772}  
extend the methods of \cite{MR0311856} 
to study optimal constants in log-Sobolev inequalities on ${\mathbb R}$.
Because log-Sobolev inequalities are preserved by the formation of products of independent
distributions (i.e. tensorization), their results yield log-Sobolev inequalities for product measures.
Their results have been refined by \cite{MR2052235} who go on in \cite{MR2430612} to study modified
log-Sobolev inequalities.
 \cite{MR3961231}  
use the ``two-sided'' Hardy inequality given by (\ref{SW}) to give an alternative proof of Cheeger's inequality.
Applications of the Hardy inequality (\ref{Hardy}) with $F$ continuous to semiparametric models for survival analysis
were given by
 \cite{MR999013}  
and  \cite{MR1623559}. 
As noted in Sections~\ref{sec:Mr}, \ref{sec:Ci}, \ref{sec:CarlemanPolyaKnopp}, and \ref{sec:MartingalesAndHoperators},
these results yield martingale connections with the operators
$\overline{H}_F $ and $\overline{H}_F^*$.

There has been some related work on Hardy type inequalities with similar unification (of continuous and
discrete cases) as an explicit goal:
for example, see
 \cite{MR1920123} 
and \cite{MR2376257},   page 45.    
Li and Mao (2020), page 257 and 258, refer to Prohorov (2008). They all study general measures.

What about related work on formulating probabilistic versions of Hardy type inequalities? 
We have not found any results in this direction.
Despite the many applications of Hardy and Muckenhoupt type inequalities in probability
theory over the past 30 years,   we are unaware of any explicit mention of these inequalities
in terms of random variables.   It seems to us that these inequalities should be better known
in both the probability and statistics communities, and the probability versions may stimulate
both further applications
and further theoretical developments.
In any case, it seems to be  worthwhile
to understand when several different formulations can be unified.

In Section~\ref{sec:MartingalesAndHoperators}  
we sketched the connection between the operators $H_F^*$ and $\overline{H}_F^*$ appearing in our
probabilistic version of Copson's dual inequality and a simple counting process martingale.
The key functions  $\overline{\Lambda}_F (x)$ and $\Lambda_F(x)$ appearing in those operators
(recall (\ref{CumulativeHazardFunctions}) for the explicit definitions) play an extremely important role in survival analysis and reliability
theory.   Also note that they do not appear without the probabilistic perspective adopted in our approach.
In the Appendix (Section~\ref{sec:Appendix}) we discuss how these functions arise in connection with left and right
censored survival data.

\section{Appendix}
\label{sec:Appendix}  %
Right and Left  censored data:   the forward  and reverse hazard functions.

Here we go further with the discussion concerning the forward and backward hazard
functions connected with our random variable versions of the Copson inequalities.

\subsection{Censored survival data:  from the right and  from the left}
Suppose that  $X_1, \ldots , X_n$ are i.i.d. survival times with d.f. $F$ on $[0,\infty) $.
Furthermore, suppose that $Y_1, \ldots , Y_n $ are i.i.d. censoring times
(independent of $X_1, \ldots , X_n$)
with distribution function $G$.
Unfortunately we do not get to observe the $X_i$'s.  Instead, for each individual we observe
$$
(Z_i , \delta_i) \equiv ( X_i \wedge Y_i , \delta_i) \equiv (X_i \wedge Y_i , {\bf 1}_{[X_i \le Y_i]} ).
$$
Nevertheless, our goal is to estimate the cumulative hazard function
$$
\overline{\Lambda}_F (t) = \int_{[0,t]} (1-F(s-))^{-1} d F(s)
$$
and the survival function $1-F$ nonparametrically.
Actually, once we have an estimator $\widehat{\overline{\Lambda}}_{F,n}$ of $\overline{\Lambda}_F$,
then estimation of $1-F$ (and hence also $F$) is immediate since
$$
1-F(t) = \exp ( - \overline{\Lambda}_c (t)) \prod_{s\le t} (1 - \Delta \overline{\Lambda} (s) ),
$$
where $\Delta \overline{\Lambda}(s)  \equiv \overline{\Lambda} (s) - \overline{\Lambda} (s-)  $
and
$ \overline{\Lambda}_c (t) \equiv \Lambda (t) - \sum_{s \le t}  \Delta \Lambda (s) $.
This is the setting of (random, right) - censored survival data, and the (nonparametric) maximum
likelihood estimators of $\overline{\Lambda} $ and $1-F$ are the famous Nelson-Aalen
estimators $\widehat{\overline{\Lambda}}$ of $\Lambda$ and Kaplan-Meier estimator $1- \widehat{F}_n$  of $1-F$.
This is the random censorship version of right-censored survival data.
For treatments of fixed (i.e. deterministic) censoring times,
see  \cite{MR1089429} 
and \cite{MR0411040}.  

Before discussing right-censoring further, suppose instead that we observe
$$
(W_i , \gamma_i) \equiv ( U_i \vee V_i , {\bf 1}_{[U_i \ge V_i]} )
$$
where the $U_i$'s are i.i.d. with d.f. $F$, and the $V_i$'s are i.i.d. $G$
(and independent of the $U_i$'s).     The goal again is to estimate
the (reverse or backwards) cumulative hazard function
$\Lambda _F (t) \equiv \int_{[t,\infty)}  dF (s) / F(s) $ and the d.f.  $F$.
This is left-censored survival data.  Note that
$\Lambda_F $ is the function which arose naturally  in the random variable
version of Copson's inequality in~Section~\ref{sec:MartingalesAndHoperators}.
A famous example
of left-censored data is the data which arose in a study of  the
descent times of baboons in the Amboseli Reserve, Kenya.
See
\cite{Wagner-Altmann:73},   
\cite{Ware-DeMets:76},  
\cite{MR688619},    
\cite{MR851047}.   

In this study the $U_i$'s
represent the times when the baboons descended from the trees in the morning while
the $V_i$'s represent the times at which the investigators arrived at the study
site.  If a baboon descended before its observer arrived at the study site,
then that baboon's $U_i$ is regarded as being ``left - censored''.   Again the
goal is nonparametric estimation of the d.f. of the $U_i$'s.

In this setting, once we have an estimator $\widehat{\Lambda}_{F,n}$ of $\Lambda = \Lambda_F$,
then estimation of $F$   is immediate since
$$
F(t) = \exp \left ( - \Lambda_c (t) \right ) \prod_{s\ge t} \Delta \Lambda (s)
$$
where
$$
\Delta \Lambda(s)  \equiv \Lambda(s) - \Lambda(s-), \quad \Lambda_c (t) \equiv \Lambda (t) - \sum_{s \ge t} \Delta \Lambda (s) .
$$

\subsection{Nonparametric estimation for right or left censored survival data}

First the classical and frequently occurring  censoring from the right.
To see that $\overline{\Lambda}_F$ and $1-F$ can be estimated nonparametrically
from the observed data, consider the following empirical distributions:
\begin{eqnarray*}
\HH_n^{uc} (t)
         & = & \PP_n (\delta {\bf 1}_{[Z \le t]}) = n^{-1} \sum_{i=1}^n \delta_i {\bf 1}_{[Z_i \le t]} , \\
\HH_n^{c} (t)
         & =  &  \PP_n ((1-\delta) {\bf 1}_{[Z \le t]}) = n^{-1} \sum_{i=1}^n (1-\delta_i) {\bf 1}_{[Z_i \le t]}, \\
\HH_n (t) & = & \PP_n {\bf 1}_{[Z \le t]} = n^{-1} \sum_{i=1}^n {\bf 1}_{[Z_i \le t]}
\end{eqnarray*}
where ``$uc$'' stands for ``uncensored'' observations and ``$c$'' stands for ``censored'' observations.
By the strong law of large numbers,
\begin{eqnarray*}
\HH_n^{uc} (t) & \rightarrow_{a.s.} & E ( \delta {\bf 1}_{[Z \le t]} )  = \int_{[0,t]} (1- G(s-)) d F(s) = H^{uc} (t) , \\
\HH_n^{c} (t) & \rightarrow_{a.s.} & E ( (1-\delta) {\bf 1}_{[Z \le t]} )  = \int_{[0,t]} (1- F(s)) d G(s) = H^{c} (t) , \\
\HH_n (t) & \rightarrow_{a.s.} &  P( Z \le t ) = 1 - (1- F(t))(1-G(t)) = H(t) . \\
\end{eqnarray*}
Now note that
\begin{eqnarray*}
\overline{\Lambda}_F (t)
 = \int_{[0,t]} \frac{1}{1-F_{-}}dF = \int_{[0,t]} \frac{1-G_{-} }{(1-G_{-} )(1-F_{-} )} d F  
 =  \int_{[0,t]} \frac{1}{1- H(s-)} d H^{uc} (s) ,
\end{eqnarray*}
so we can estimate $\overline{\Lambda}_F$ by
\begin{eqnarray*}
\widehat{\overline{\Lambda}}_n (t) \equiv \int_{[0,t]} \frac{1}{ 1- \HH_n (s -)}  d \HH_n^{uc} (s) .
\end{eqnarray*}
Then
$1 - \widehat{F}_n (t) = \prod_{s \le t} (1 - \Delta \widehat{\overline{\Lambda}}_n (s) ) $
is the  \cite{MR93867}  
estimator of $1-F$.

Now for estimation in the presence of censoring from the left.
To see that $\overline{\Lambda}_F$ and $\Lambda_F$ can be estimated nonparametrically
from the observed (left-censored) data, consider the following empirical distributions:
\begin{eqnarray*}
\HH_n^{uc} (t) & = & \PP_n (\gamma {\bf 1}_{[W \le t]}) = n^{-1} \sum_{i=1}^n \gamma_i {\bf 1}_{[W_i \le t]} , \\
\HH_n^{c} (t)  & =  & \PP_n ((1-\gamma) {\bf 1}_{[W \le t]}) = n^{-1} \sum_{i=1}^n (1-\gamma_i) {\bf 1}_{[W_i \le t]}, \\
\HH_n (t) & = & \PP_n {\bf 1}_{[W  \le t ]} .
\end{eqnarray*}
Now
\begin{eqnarray*}
\HH_n^{uc} (t) & \rightarrow_{a.s.} & E ( \gamma {\bf 1}_{[W \le t]} )  = \int_{[0,t]} G(s) d F(s), \\
\HH_n^{c} (t) & \rightarrow_{a.s.} & E ( (1-\gamma) {\bf 1}_{[W \le t]} )  = \int_{[0,t]} F(s-) d G(s), \\
\HH_n (t) & \rightarrow_{a.s.} &  P( W \le t ) = F(t)G(t)  . \\
\end{eqnarray*}
Now note that
\begin{eqnarray*}
\Lambda_F (t)
 =  \int_{[s \ge t]} \frac{1}{F}  dF = \int_{[s\ge t]} \frac{G }{G F} d F
 =  \int_{[s \ge t]} \frac{1}{H(s)} d H^{uc} (s) ,
\end{eqnarray*}
so we can estimate the ``backwards''  Nelson-Aalen hazard function $\Lambda_F$ by
\begin{eqnarray*}
\widehat{ \Lambda}_n (t) \equiv \int_{[s\ge t]} \frac{1}{ \HH_n (s)}  d \HH_n^{uc} (s) .
\end{eqnarray*}
Then
$\widehat{F}_n (t) = \prod_{s \ge t} \Delta \widehat{\Lambda}_n (s)  $
 is the ``reverse'' or ``backwards'' Kaplan - Meier estimator of  $F$;
 see e.g. \cite{Ware-DeMets:76}  
 and \cite{MR688619},   
 \cite{MR851047}.  
For more on left-censoring,
the data in the baboon study, and a plot of the
resulting backwards Kaplan-Meier estimator,  see
 \cite{MR1198884}, pages 24, 162-165, and 273-274.
 \medskip


\par\noindent
{\bf Acknowledgement:}
The authors owe thanks to Peter Bickel for pointing out the
relevance of Hardy's inequality in the context of information bounds
for survival analysis models.
Thanks also go to Adrien Saumard for  several helpful comments.
Finally we thank the AE for suggesting that we should also include Muckenhoupt's inequality in the
current study.


\bibliographystyle{ims}
\bibliography{Klaassen-HardyIneq}

\end{document}